\documentclass{amsart}

\usepackage[utf8]{inputenc} 
\usepackage[T1]{fontenc} 
\usepackage[english]{babel} 
\usepackage{mathrsfs}
\numberwithin{equation}{section}

\theoremstyle{plain} 
\newtheorem{theorem}{Theorem}[section]
\newtheorem{lemma}[theorem]{Lemma}
\newtheorem{corollary}[theorem]{Corollary}

\theoremstyle{remark} 
\newtheorem*{remark}{Remark}

\theoremstyle{definition}
\newtheorem*{definition}{Definition}
\newtheorem{problem}{Problem}

\DeclareMathOperator{\mre}{Re}
\DeclareMathOperator{\mim}{Im}
\DeclareMathOperator{\Arg}{Arg}

\begin{document} 
\title[A mean counting function and compact composition operators]{A mean counting function for Dirichlet series and compact composition operators}
\date{\today} 

\author{Ole Fredrik Brevig}
\address{Department of Mathematics, University of Oslo, 0851 Oslo, Norway} 
\email{obrevig@math.uio.no}

\author{Karl-Mikael Perfekt}
\address{Department of Mathematics and Statistics, University of Reading, Reading RG6 6AX, United Kingdom}
\email{k.perfekt@reading.ac.uk}

\begin{abstract}
	We introduce a mean counting function for Dirichlet series, which plays the same role in the function theory of Hardy spaces of Dirichlet series as the Nevanlinna counting function does in the classical theory. The existence of the mean counting function is related to Jessen and Tornehave's resolution of the Lagrange mean motion problem. We use the mean counting function to describe all compact composition operators with Dirichlet series symbols on the Hardy--Hilbert space of Dirichlet series, thus resolving a problem which has been open since the bounded composition operators were described by Gordon and Hedenmalm. The main result is that such a composition operator is compact if and only if the mean counting function of its symbol satisfies a decay condition at the boundary of a half-plane.
\end{abstract}


\maketitle

\section{Introduction}
Let $\mathscr{H}^2$ denote the Hardy space of Dirichlet series $f(s)=\sum_{n\geq1} a_n n^{-s}$ with square summable coefficients, 
\[\|f\|_{\mathscr{H}^2}^2 = \sum_{n=1}^\infty |a_n|^2 < \infty.\]
By the Cauchy--Schwarz inequality, it is easily verified that $\mathscr{H}^2$ is a space of analytic functions in the half-plane $\mathbb{C}_{1/2}$, where $\mathbb{C}_\theta = \left\{s \in \mathbb{C}\,:\,\mre{s}>\theta\right\}$. 

Therefore, if $\varphi\colon\mathbb{C}_{1/2}\to\mathbb{C}_{1/2}$ is an analytic function, the composition $\mathscr{C}_\varphi f = f\circ\varphi$ defines an analytic function in $\mathbb{C}_{1/2}$ for every $f\in\mathscr{H}^2$. However, not every \emph{symbol} $\varphi$ yields a bounded composition operator $\mathscr{C}_\varphi \colon \mathscr{H}^2 \to \mathscr{H}^2$. In their seminal paper \cite{GH99}, Gordon and Hedenmalm demonstrated that $\mathscr{C}_\varphi$ defines a bounded composition operator on $\mathscr{H}^2$ if and only if $\varphi$ belongs to the Gordon--Hedenmalm class $\mathscr{G}$.

\begin{definition}
	The Gordon--Hedenmalm class, denoted $\mathscr{G}$, consists of the functions $\varphi\colon \mathbb{C}_{1/2}\to\mathbb{C}_{1/2}$ of the form
	\[\varphi(s) = c_0 s + \sum_{n=1}^\infty c_n n^{-s} = c_0s + \varphi_0(s),\]
	where $c_0$ is a non-negative integer and the Dirichlet series $\varphi_0$ converges uniformly in $\mathbb{C}_\varepsilon$ for every $\varepsilon>0$ and satisfies the following mapping properties:
	\begin{enumerate}
		\item[(a)] If $c_0=0$, then $\varphi_0(\mathbb{C}_0)\subseteq \mathbb{C}_{1/2}$.
		\item[(b)] If $c_0\geq1$, then either $\varphi_0(\mathbb{C}_0)\subseteq \mathbb{C}_0$ or $\varphi_0\equiv0$.
 	\end{enumerate}
\end{definition}

Ever since \cite{GH99} appeared in 1999, one of the field's central open problems has been to classify the symbols generating compact composition operators on $\mathscr{H}^2$. We refer to the surveys \cite[Prob.~4]{Hedenmalm04} and \cite[Prob.~3.3]{SS16} for a discussion of the problem and to the papers \cite{Bailleul15,Bayart02,Bayart03,BB17,FQV04,QS15} for some partial results. It is important to know that the two cases of the Gordon--Hedenmalm class are fundamentally different. In particular, the associated symbols feature almost periodic behavior only in case (a). The notion of almost periodicity will play an essential role in this paper.

Up to now, the problem has been open in both cases (a) and (b), although the latter case has seen much more progress (which we will briefly review in Section~\ref{subsec:b}). In case (a), the only non-trivial results have pertained to specific polynomial symbols \cite{Bayart03,BB17,FQV04,QS15} and periodic symbols \cite{QS15}. The main purpose of the present paper is to completely resolve the compactness problem in case (a). To simplify certain statements below, we will therefore let $\mathscr{G}_0$ denote subclass of $\mathscr{G}$ where $c_0=0$.

We have two sources of inspiration. The first is the description of compact composition operators on the Hardy space of the unit disc $\mathbb{D}=\{z\in\mathbb{C}\,:\,|z|<1\}$ due to Shapiro \cite{Shapiro87}. Recall in this setting that Littlewood's subordination principle \cite{Littlewood25} implies that any analytic function $\phi \colon \mathbb{D} \to \mathbb{D}$ generates a bounded composition operator on $H^2(\mathbb{D})$. For $\xi\neq\phi(0)$, consider the \emph{Nevanlinna counting function}
\begin{equation} \label{eq:nevanlinnadisc}
	N_\phi(\xi) = \sum_{z \in \phi^{-1}(\{\xi\})} \log{\frac{1}{|z|}}.
\end{equation}
Shapiro's result states that the composition operator generated by $\phi$ on $H^2(\mathbb{D})$ is compact if and only if
\begin{equation} \label{eq:compactdisc}
	\lim_{|\xi| \to 1^{-}} \frac{N_\phi(\xi)}{\log\frac{1}{|\xi|}} = 0.
\end{equation}
In this context, we recall the Littlewood inequality: for an analytic function $\phi\colon\mathbb{D}\to\mathbb{D}$ and every $\xi\in \mathbb{D}\setminus\{\phi(0)\}$,
\begin{equation} \label{eq:littlewoodineq}
	N_\phi(\xi) \leq \log\left|\frac{1-\overline{\xi}\phi(0)}{\xi-\phi(0)}\right|.
\end{equation}
Consequently, the limiting quantity in \eqref{eq:compactdisc} is bounded for every analytic function $\phi\colon\mathbb{D}\to\mathbb{D}$. This observation yields one way to prove the boundedness of the composition operator generated by $\phi$, via the Stanton formula, which we shall provide an analogue of in Theorem~\ref{thm:cov}. We remark also that by the maximum principle, \eqref{eq:compactdisc} is a statement about the behavior of $\phi$ near the boundary of $\mathbb{D}$.

Our second source of inspiration is Jessen and Tornehave's \cite{JT45} extraordinary resolution of the Lagrange mean motion problem, outlined in Section~\ref{sec:JT}. Let $f \not \equiv 0$ be a Dirichlet series which converges uniformly in $\mathbb{C}_\varepsilon$ for every $\varepsilon > 0$. Then $f$ is \emph{almost periodic} in every such half-plane, and therefore the equation $f(s) = 0$ has either $0$ or an infinite number of solutions $s\in\mathbb{C}_0$. Jessen and Tornehave proved that the \emph{unweighted mean counting function}
\[\mathscr{Z}_f(\sigma) = \lim_{T \to \infty} \frac{1}{2T} \sum_{\substack{s \in f^{-1}(\{0\}) \\ \,\,|\mim{s}|<T \\ \sigma<\mre{s} < \infty}} 1 \]
exists for every $\sigma > 0$. Furthermore, if $f(+\infty) \neq 0$, then $\mathscr{Z}_f$ coincides with the right-derivative of the \emph{Jessen function}
\[\mathscr{J}_f(\sigma) = \lim_{T \to \infty} \frac{1}{2T} \int_{-T}^T \log |f(\sigma+it)| \, dt.\]
The convexity of $\mathscr{J}_f$, previously demonstrated by Jessen \cite{Jessen33}, plays an important role, in particular guaranteeing that the the right-derivative exists. 

Let $f$ be a Dirichlet series which converges uniformly in $\mathbb{C}_\varepsilon$ for every $\varepsilon>0$. As an analogue of the Nevanlinna counting function \eqref{eq:nevanlinnadisc}, we introduce for $w \neq f(+\infty)$ the \emph{mean counting function}
\begin{equation} \label{eq:meancounting}
	\mathscr{M}_f(w) = \lim_{\sigma\to 0^+}\lim_{T\to\infty} \frac{\pi}{T} \sum_{\substack{s \in f^{-1}(\{w\}) \\ \,\,|\mim{s}|<T \\ \sigma<\mre{s}<\infty}} \mre{s}.
\end{equation}
Through Littlewood's lemma \cite{Littlewood24}, we can relate the mean counting function \eqref{eq:meancounting} to the Jessen function and the unweighted mean counting function. In particular, if $f$ satisfies a Nevanlinna--type condition, we will in Theorem~\ref{thm:Mexistgeneral} show that
\begin{equation} \label{eq:countingJintro}
	\mathscr{M}_f(w) = \lim_{\sigma \to 0^+}  \mathscr{J}_{f- w}(\sigma) - \log |f(+\infty) - w|.
\end{equation}
Exploiting the mapping properties of a symbol $\varphi \in \mathscr{G}_0$, equation \eqref{eq:countingJintro} will lead us to our first main result.
\begin{theorem} \label{thm:mean}
	Suppose that $\varphi \in \mathscr{G}_0$. The mean counting function $\mathscr{M}_\varphi$ exists for every $w \in \mathbb{C}_{1/2}\setminus\{\varphi(+\infty)\}$ and enjoys the point-wise estimate  
	\begin{equation} \label{eq:littlewoodineqM}
		\mathscr{M}_\varphi(w) \leq \log\left|\frac{\overline{w}+\varphi(+\infty)-1}{w-\varphi(+\infty)}\right|.
	\end{equation}
\end{theorem}

To note the analogy of \eqref{eq:littlewoodineqM} with the Littlewood inequality \eqref{eq:littlewoodineq}, we point out that the pseudo-hyperbolic distance in $\mathbb{C}_{1/2}$ between two points $w,\nu\in\mathbb{C}_{1/2}$ is given by
\[\varrho_{\mathbb{C}_{1/2}}(w,\nu) = \left|\frac{w-\nu}{\overline{w}+\nu-1}\right|,\]
while the pseudo-hyperbolic distance in $\mathbb{D}$ between two points $\xi, \nu \in \mathbb{D}$ is given by
\[\varrho_{\mathbb{D}}(\xi,\nu) = \left|\frac{\xi-\nu}{1-\overline{\xi}\nu}\right|.\]
Note also that for fixed $\nu=\varphi(+\infty)$ the upper bound of \eqref{eq:littlewoodineqM} decays like $(\mre{w}-1/2)$ when $\mre{w}\to 1/2^+$ and $\mim w$ is fixed, and like $|\mim{w}|^{-2}$ when $|\mim{w}|\to \infty$ and $\mre w$ is fixed. We refer to Lemma~\ref{lem:logest} for the precise behavior.

It is natural to next ask which $\varphi\in\mathscr{G}_0$ attain the upper bound \eqref{eq:littlewoodineqM}. Fix $\nu\in\mathbb{C}_{1/2}$ and consider 
\begin{equation} \label{eq:varphinu}
	\varphi_\nu(s) = \frac{\nu + (1-\nu) 2^{-s}}{1+2^{-s}}.
\end{equation}
Since this symbol is periodic, it is easy to establish that
\begin{equation} \label{eq:Mvarphinu}
	\mathscr{M}_{\varphi_\nu}(w) = \log\left|\frac{\overline{w}+\nu-1}{w-\nu}\right|
\end{equation}
for every $w \in \mathbb{C}_{1/2}\setminus\{\nu\}$ (see Section~\ref{sec:prelim}). In fact, we can completely classify the symbols $\varphi$ attaining equality in \eqref{eq:littlewoodineqM} for just one $w\in\mathbb{C}_{1/2}$. To state our result, recall that every $\varphi \in \mathscr{G}_0$ has \emph{generalized boundary values} $\varphi^\ast(\chi)$ for almost every $\chi \in \mathbb{T}^\infty$ (see \cite[Sec.~2]{BP20} or Corollary~\ref{cor:varphiboundary} below).  

\begin{theorem} \label{thm:sharpiro}
	Suppose that $\varphi$ is in $\mathscr{G}_0$ and let $\nu = \varphi(+\infty)\in\mathbb{C}_{1/2}$. The following are equivalent:
	\begin{enumerate}
		\item[(i)] $\mre{\varphi^\ast(\chi)}=1/2$ for almost every $\chi \in \mathbb{T}^\infty$.
		\item[(ii)] $\mathscr{M}_\varphi(w) = \log\big|\frac{\overline{w}+\nu-1}{w-\nu}\big|$ for quasi-every $w\in\mathbb{C}_{1/2}$. 
		\item[(iii)] $\mathscr{M}_\varphi(w) = \log\big|\frac{\overline{w}+\nu-1}{w-\nu}\big|$ for one $w\in\mathbb{C}_{1/2}$. 
	\end{enumerate}
\end{theorem}

Theorem~\ref{thm:sharpiro} should be compared with the analogous statement \cite[Sec.~4.2]{Shapiro87} for the Nevanlinna counting function \eqref{eq:nevanlinnadisc} and the Littlewood--type inequality \eqref{eq:littlewoodineq}. A key component in the proof of the latter statement is Frostman's theorem for inner functions in the unit disc, leading us to develop some of the corresponding theory for Dirichlet series.

For $f$ belonging to the Hardy space of Dirichlet series $\mathscr{H}^p$, $0 < p < \infty$, we shall therefore consider a generalization of the Jessen function,
\begin{equation} \label{eq:introconvex}
	\int_{\mathbb{T}^\infty} \log |f_\chi(\sigma)| \, dm_\infty(\chi).
\end{equation}
Here $f_\chi$ denotes the \emph{vertical limit function} of $f$ for $\chi\in\mathbb{T}^\infty$, and $f_\chi(0)$ denotes the generalized boundary value $f^\ast(\chi)$ (see Sections~\ref{sec:prelim} and \ref{sec:convexity} for these two notions). Our first result, Theorem~\ref{thm:logintdown}, is that \eqref{eq:introconvex} is a convex non-increasing function of $\sigma \geq 0$. This generalizes the convexity of the Jessen function, since we do not assume that $f$ converges uniformly in $\mathbb{C}_{\theta}$ for any $\theta \leq 1/2$. We also obtain a new proof of the fact that $\log |f^\ast| \in L^1(\mathbb{T}^\infty)$ for $f \in \mathscr{H}^p$, a result which recently appeared in \cite{AOS19}.

Based on our study of \eqref{eq:introconvex} we will then prove a Frostman theorem in Rudin's form \cite{Rudin67}. We describe only a special case of the result here, deferring the more general statement to Theorem~\ref{thm:rudinfrostman}. Recall that $f \in \mathscr{H}^2$ is \emph{inner} if $|f^\ast(\chi)| = 1$ for almost every $\chi \in \mathbb{T}^\infty$. The result is that if $f$ is inner, then for quasi-every $\alpha \in \mathbb{D}$ (in the sense of logarithmic capacity), we have that
\[\lim_{\sigma\to0^+} \lim_{T\to\infty} \frac{1}{2T}\int_{-T}^T \log|\phi_\alpha \circ f(\sigma+it)|\,dt = 0,\]
where  $\phi_\alpha$ is the automorphism of $\mathbb{D}$ given by $\phi_\alpha(z) = \frac{\alpha-z}{1-\overline{\alpha}z}$. Note here that that every inner function $f \in \mathscr{H}^2$ converges uniformly in  $\mathbb{C}_{\varepsilon}$ for every $\varepsilon > 0$, by Bohr's theorem (Lemma~\ref{lem:bohr}). 

A slightly different corollary of our general Frostman theorem will be employed to prove the implication (i) $\implies$ (ii) of Theorem~\ref{thm:sharpiro}. In Theorem~\ref{thm:innertfae}, we will also give a version of Theorems~\ref{thm:mean} and \ref{thm:sharpiro} for inner functions in $\mathscr{H}^2$, which is a direct generalization of the result in \cite[Sec.~4.2]{Shapiro87}.

Our third main result is the analogue of the Stanton formula, which arises from a non-injective change of variables in the Littlewood--Paley formula for the $\mathscr{H}^2$-norm (Lemma~\ref{lem:curlyLP}). The classical Stanton formula may be found for example in \cite[Sec.~10.3]{Shapiro93}. To prove our formula, we need a set of technical estimates collected in Lemma~\ref{lem:ThetaLsub}, in addition to the existence of the mean counting function $\mathscr{M}_\varphi$ from Theorem~\ref{thm:mean}.
\begin{theorem} \label{thm:cov}
	Let $\mathscr{M}_\varphi$ be the mean counting function \eqref{eq:meancounting} for some $\varphi\in\mathscr{G}_0$. Then 
	\begin{equation} \label{eq:stanton}
		\|\mathscr{C}_\varphi f\|_{\mathscr{H}^2}^2 = |f(\varphi(+\infty))|^2 + \frac{2}{\pi}\int_{\mathbb{C}_{1/2}} |f'(w)|^2 \mathscr{M}_\varphi(w)\,dw
	\end{equation}
	for every $f\in\mathscr{H}^2$.
\end{theorem}

We will apply Theorem~\ref{thm:cov} to prove our final main result.

\begin{theorem} \label{thm:compact}
	Let $\varphi \in \mathscr{G}_0$ and let $\mathscr{M}_\varphi$ denote the mean counting function \eqref{eq:meancounting}. Then $\mathscr{C}_{\varphi} \colon \mathscr{H}^2 \to \mathscr{H}^2$ is compact if and only if
	\begin{equation*} 
		\lim_{\mre w \to \frac{1}{2}^+} \frac{\mathscr{M}_{\varphi}(w)}{\mre w - 1/2} = 0.
	\end{equation*}
\end{theorem}
Given Theorem~\ref{thm:cov} and the submean value property of the mean counting function (Lemma~\ref{lem:submean}), it is straightforward to prove that the vanishing condition is necessary for compactness, by using the reproducing kernels of $\mathscr{H}^2$ in a standard manner. The sufficiency relies on Carleson measure techniques and the fact that the Littlewood-type inequality \eqref{eq:littlewoodineqM} actually produces more than sufficient decay as $|\mim{w}|\to\infty$. 

\subsection*{Organization} In Section~\ref{sec:prelim} we discuss the notions of almost periodicity and vertical limits of Dirichlet series, and examine the mean counting function in the simple case of a periodic symbol. Furthermore, we establish certain technical estimates needed for the proof of Theorem~\ref{thm:cov}. In Section~\ref{sec:convexity} we outline the foundations of $\mathscr{H}^p$-theory and establish the convexity of the logarithmic integrals \eqref{eq:introconvex}. In Section~\ref{sec:jessen} we discuss the Jessen function, a Nevanlinna-type class of Dirichlet series, and prove our Frostman theorem. In Section~\ref{sec:JT} we briefly outline the work of Jessen and Tornehave \cite{JT45}. In Section~\ref{sec:mean} we prove that the mean counting function exists and establish Theorems~\ref{thm:mean} and \ref{thm:sharpiro}. In Section~\ref{sec:compact} we prove Theorems~\ref{thm:cov} and \ref{thm:compact}. In Section~\ref{sec:further} we discuss some interesting open problems, as well as outline the state-of-the-art for composition operators associated with case (b) of the Gordon--Hedenmalm class.

\subsection*{Acknowledgments} The authors are grateful to Titus Hilberdink and Kristian Seip for some useful discussions. The research of K.-M. Perfekt was supported by the UK Engineering and Physical Sciences Research Council, grant EP/S029486/1.

\section{Preliminaries and applications of the classical counting function} \label{sec:prelim}
Consider a Dirichlet series $f(s)=\sum_{n\geq1} a_n n^{-s}$. The \emph{abscissa of convergence} $\sigma_{\operatorname{c}}(f)$ is defined as the infimum of all real numbers $\theta$ such that $f$ converges in the half-plane $\mathbb{C}_\theta$, and $\sigma_{\operatorname{a}}(f)$ is the similarly defined \emph{abscissa of absolute convergence}. As mentioned in the introduction, we have $\sigma_{\operatorname{a}}(f)\leq 1/2$ for every $f \in \mathscr{H}^2$. This is easily seen to be optimal, for instance by considering
\[f(s) = \sum_{n=2}^\infty \frac{1}{\sqrt{n}(\log{n})} n^{-s},\]
which evidently satisfies $\sigma_{\operatorname{c}}(f)=\sigma_{\operatorname{a}}(f)=1/2$. 

The \emph{abscissa of uniform convergence} $\sigma_{\operatorname{u}}(f)$ is the infimum of all $\theta$ such that the Dirichlet series $f$ converges uniformly in $\overline{\mathbb{C}_{\theta}}$. We recall from the introduction that if $\varphi\in\mathscr{G}_0$, then $\sigma_{\operatorname{u}}(\varphi)\leq 0$ (see also \cite[Sec.~3]{QS15}). The well-known upper bound
\begin{equation} \label{eq:au12}
	\sigma_{\operatorname{a}}(f)-\sigma_{\operatorname{u}}(f)\leq \frac{1}{2}
\end{equation}
follows at once from the Cauchy--Schwarz inequality and the Cahen--Bohr formulas for the relevant abscissas (see e.g. \cite[Thm.~4.2.1]{QQ13}). A deep result of Bohnenblust and Hille \cite{BH31} is that the upper bound in \eqref{eq:au12} is sharp. We shall also have use of the following result, often called Bohr's theorem \cite{Bohr13}. 

\begin{lemma} \label{lem:bohr}
	Suppose that $f$ is a somewhere convergent Dirichlet series. Let $\sigma_{\operatorname{b}}(f)$ denote the infimum of all $\theta$ such that $f$ may be extended (by analytic continuation if necessary) to a bounded analytic function in $\mathbb{C}_\theta$. Then $\sigma_{\operatorname{u}}(f)=\sigma_{\operatorname{b}}(f)$.
\end{lemma}

A function $h \colon \mathbb{R} \to \mathbb{C}$ is \emph{almost periodic} if there for every $\varepsilon > 0$ exists a relatively dense set of real numbers $\tau$ such that
\[\sup_{t \in \mathbb{R}} |h(t+\tau) - h(t)| \leq \varepsilon.\]
A well-known equivalent definition, due to Bochner, is that the set $\{h(\cdot + \tau)\}_{\tau \in \mathbb{R}}$ is relatively compact in the topology of uniform convergence. If $h$ is almost periodic, then the mean value
\[\lim_{T \to \infty} \frac{1}{2T} \int_{-T}^T h(t) \, dt\]
exists. For $\tau \in \mathbb{R}$, let $T_\tau$ denote the vertical translation $T_\tau f(s) =f(s+i\tau)$. We say that a function $f \colon \overline{\mathbb{C}_\theta} \to \mathbb{C}$ is almost periodic if there for each $\varepsilon>0$ exists a relatively dense set of real numbers $\tau$ such that
\[\sup_{\mre{s}\geq \theta} |T_\tau f(s)-f(s)|\leq\varepsilon.\]
It is well-known that if $f(s) = \sum_{n\geq1} a_n n^{-s}$ satisfies $\sigma_{\operatorname{u}}(f)\leq0$, then $f$ is almost periodic in the half-plane $\overline{\mathbb{C}_\theta}$ for every $\theta>0$ (see e.g.~\cite[Sec.~1.5]{QQ13}). In this case,
\[\lim_{T \to \infty} \frac{1}{2T} \int_{-T}^T f(\sigma + it) \, dt = a_1\]
for every $\sigma >0$.

Let us now discuss how uniformly convergent Dirichlet series have \emph{vertical limit functions}. Let $\mathbb{T}^\infty$ denote the countably infinite Cartesian product of the torus,
\[\mathbb{T}=\{z\in\mathbb{C}\,:\,|z|=1\},\]
forming a compact commutative group under coordinate-wise multiplication. The Haar measure of $\mathbb{T}^\infty$, denoted $m_\infty$, is the countably infinite product measure generated by the normalized Lebesgue arc length measure of $\mathbb{T}$, denoted $m$. For
\[\chi = (\chi_1,\chi_2,\ldots)\in\mathbb{T}^\infty,\]
define the \emph{character} $\chi \colon \mathbb{N}\to\mathbb{T}$ as the completely multiplicative function defined by $\chi(p_j)=\chi_j$, where $(p_j)_{j\geq1}$ denotes the increasing sequence of prime numbers. For a Dirichlet series $f(s)=\sum_{n\geq1} a_n n^{-s}$ and a character $\chi$, the vertical limit function $f_\chi$ is defined as 
\[f_\chi(s) = \sum_{n=1}^\infty a_n \chi(n) n^{-s}\]
Note that the vertical translation $T_\tau f$ corresponds to the vertical limit function $f_\chi$ of the character $\chi(n) = n^{-i\tau}$. 

Let $f$ be a Dirichlet series which is uniformly convergent in $\overline{\mathbb{C}_\theta}$. By e.g.~\cite[Lem.~2.4]{HLS97}, the vertical limit functions $f_\chi$ are precisely the functions obtained as uniform limits
\begin{equation} \label{eq:chitauk}
	f_\chi(s) = \lim_{k\to\infty} T_{\tau_j}f(s)
\end{equation}
in $\overline{\mathbb{C}_\theta}$, where $(\tau_j)_{j\geq1}$ is a sequence of real numbers.

By almost periodicity and Lemma~\ref{lem:bohr}, we see from \eqref{eq:chitauk} that $\sigma_{\operatorname{u}}(f)=\sigma_{\operatorname{u}}(f_\chi)$ for every Dirichlet series $f$ and $\chi \in \mathbb{T}^\infty$. However, the abscissa of convergence is not invariant under vertical limits (see \cite[Sec.~2.2]{Bayart02} or \cite[Sec.~4.2]{HLS97}). Specifically, if $f\in\mathscr{H}^2$, then $\sigma_{\operatorname{c}}(f_\chi)\leq0$ for almost every $\chi \in \mathbb{T}^\infty$. Moreover, the \emph{generalized boundary value}
\begin{equation} \label{eq:genboundary}
	f^\ast(\chi) = \lim_{\sigma\to0^+} f_\chi(\sigma)
\end{equation}
exists for almost every $\chi \in \mathbb{T}^\infty$ and $\|f^\ast\|_{L^2(\mathbb{T}^\infty)}=\|f\|_{\mathscr{H}^2}$.

Under stronger assumptions, which hold in the context of Theorem~\ref{thm:cov}, it is possible to compute the $\mathscr{H}^2$-norm as an $L^2$-average near the imaginary axis. Namely, we have Carlson's formula \cite{Carlson52}, which states that if $f\in\mathscr{H}^2$ and $\sigma_{\operatorname{u}}(f)\leq0$, then 
\begin{equation} \label{eq:carlson2}
	\|f\|_{\mathscr{H}^2}^2 = \lim_{\sigma\to0^+} \lim_{T\to\infty} \frac{1}{2T} \int_{-T}^T |f(\sigma+it)|^2 \,dt.
\end{equation}

The starting point of this paper is to make note of the Littlewood--Paley formula that follows from \eqref{eq:carlson2}. We include a short proof for the benefit of the reader, although this formula is certainly known to experts. See for example \cite[Prop.~2]{Bayart02} for a similar formula which has relevance to the compactness question for symbols $\varphi$ associated with case (b) of the Gordon--Hedenmalm class.

\begin{lemma} \label{lem:curlyLP}
	Suppose that $f\in\mathscr{H}^2$ and that $\sigma_{\operatorname{u}}(f)\leq 0$. Then
	\begin{equation} \label{eq:curlyLP}
		\|f\|_{\mathscr{H}^2}^2 = |f(+\infty)|^2 + \lim_{\sigma_0 \to 0^+} \lim_{T\to\infty} \frac{2}{T} \int_{\sigma_0}^\infty \int_{-T}^T |f'(\sigma+it)|^2 \,dt \, \sigma d\sigma.
	\end{equation}
\end{lemma}

\begin{proof}
	Suppose first that $f(s) = \sum_{n\geq1} a_n n^{-s}$  is a Dirichlet polynomial and let $\sigma_0 < \sigma < \infty$. Exchanging the order of summation and integration, we see that
	\[\lim_{T\to\infty} \frac{2}{T} \int_{-T}^T |f'(\sigma+it)|^2 \,dt = 4\sum_{n=2}^\infty \frac{|a_n|^2 (\log{n})^2}{n^{2\sigma}}.\]
	In fact, this is \eqref{eq:carlson2} applied to $g(s)=f'(s+\sigma)$. Furthermore, 
	\[\lim_{T\to\infty} \frac{2}{T} \int_{\sigma_0}^\infty \int_{-T}^T |f'(\sigma+it)|^2 \,dt \, \sigma d\sigma = 4 \sum_{n=2}^\infty |a_n|^2 \int_{\sigma_0}^\infty \frac{(\log{n})^2}{n^{2\sigma}} \, \sigma\,d\sigma. \]
	By approximating with partial sums, this last formula remains valid for $f \in \mathscr{H}^2$ such that $\sigma_{\operatorname{u}}(f)\leq 0$. The identity \eqref{eq:curlyLP} now follows upon using the monotone convergence theorem to let $\sigma_0 \to 0^+$,  computing that
	\[\int_{0}^\infty \frac{(\log{n})^2}{n^{2\sigma}}\, \sigma\,d\sigma = \frac{1}{4},\]
	and recalling that $a_1=f(+\infty)$. 
\end{proof}

To obtain the change of variable formula of Theorem~\ref{thm:cov}, we are going to apply the Littlewood--Paley formula \eqref{eq:curlyLP} to the function $f \circ \varphi$. Note that $\sigma_{\operatorname{u}}(f \circ \varphi) \leq 0$ whenever $f \in \mathscr{H}^2$ and $\varphi \in \mathscr{G}_0$, by the maximum principle and Lemma~\ref{lem:bohr}. By a non-injective change of variables, we find, for $\sigma_0>0$ and $T>0$, that
\begin{equation} \label{eq:covT}
	\frac{2}{T} \int_{\sigma_0}^\infty \int_{-T}^T |(f\circ\varphi)'(\sigma+it)|^2 \,dt \, \sigma d\sigma = \frac{2}{\pi}\int_{\mathbb{C}_{1/2}} |f'(w)|^2 \mathscr{M}_\varphi(w,\sigma_0,T) \,dw,
\end{equation}
where
\begin{equation} \label{eq:meansigma0T}
	\mathscr{M}_\varphi(w,\sigma_0,T)=\frac{\pi}{T}\sum_{\substack{s \in \varphi^{-1}(\{w\}) \\ \,\,\,\,\,|\mim{s}|<T \\ \sigma_0<\mre{s}<\infty}}\mre s.
\end{equation}
If $w\neq\varphi(+\infty)$, it is clear that the equation $\varphi(s)=w$ has no solutions for sufficiently large $\mre{s}$. Consequently, \eqref{eq:meansigma0T} is a finite sum for any fixed $\sigma_0>0$ and $T>0$.

For $w \in \mathbb{C}_{1/2}$, set
\begin{equation} \label{eq:Phiw}
	\Phi_w(s) = \frac{w-\varphi(s)}{\overline{w}+\varphi(s)-1}.
\end{equation}
Clearly, $\Phi_w$ is an analytic function which maps $\mathbb{C}_0$ to $\mathbb{D}$. Moreover, since $\Phi_w$ can be expanded as a convergent Dirichlet series in some half-plane $\mathbb{C}_\theta$ it follows from Lemma~\ref{lem:bohr} that $\sigma_{\operatorname{u}}(\Phi_w)\leq 0$. This fact will not be used in the present section, but it will be important later. Noting that $\varphi(s)=w$ if and only if $\Phi_w(s)=0$, we reformulate \eqref{eq:meansigma0T} as
\begin{equation} \label{eq:meansigma0T0}
	\mathscr{M}_\varphi(w,\sigma_0,T)=\frac{\pi}{T}\sum_{\substack{s \in \Phi_w^{-1}(\{0\}) \\ \,\,\,\,\,|\mim{s}|<T \\ \sigma_0<\mre{s}<\infty}}\mre s.
\end{equation}

Let us now consider the counting function \eqref{eq:meansigma0T0} in a simple case which may be easily understood. Our analysis is inspired by the transference principle of \cite[Sec.~9]{QS15}. Suppose that $\varphi(s)=\psi(2^{-s})$ for some analytic function $\psi \colon \mathbb{D}\to \mathbb{C}_{1/2}$. In this case $\varphi$ is \emph{periodic} (with period $i 2\pi/\log{2}$). The periodicity implies that
\[\mathscr{M}_\varphi(w,\sigma_0,T) = \left(1 + O(T^{-1})\right)\,\log{2} \sum_{\substack{s \in \Phi_w^{-1}(\{0\}) \\ 0 \leq \mim{s} < 2\pi/\log{2} \\ \sigma_0<\mre{s}<\infty}}\mre s \]
Note that
\begin{equation} \label{eq:periodic}
	\log{2} \sum_{\substack{s \in \Phi_w^{-1}(\{0\}) \\ 0 \leq \mim{s} < 2\pi/\log{2} \\ \sigma_0<\mre{s}<\infty}}\mre s = \sum_{\substack{z \in \Psi_w^{-1}(\{0\}) \\ |z|<2^{-\sigma_0}}} \log{\frac{1}{|z|}},
\end{equation}
where $\Psi_w$ is the analytic self-map of $\mathbb{D}$ defined in analogy with \eqref{eq:Phiw}, that is,
\[\Psi_w(z) = \frac{w-\psi(z)}{\overline{w}+\psi(z)-1}.\]
We recognize the right hand side of \eqref{eq:periodic} as the Nevanlinna counting function \eqref{eq:nevanlinnadisc} for $\Psi_w$, with summation restricted to $z\in\mathbb{D}(0,2^{-\sigma_0})$. Therefore the mean counting function $\mathscr{M}_\varphi(w)$ exists in this case,
\[\mathscr{M}_\varphi(w) = \lim_{\sigma_0 \to 0^+} \lim_{T \to \infty} \mathscr{M}_\varphi(w,\sigma_0,T) = N_{\Psi_w}(0).\]
Furthermore, \eqref{eq:littlewoodineq} yields that 
\[\mathscr{M}_\varphi(w) \leq \log \left| \frac{1}{\Psi_w(0)}\right| = \log\left|\frac{\overline{w}+\varphi(+\infty)-1}{w-\varphi(+\infty)}\right|.\]
Hence we have proven Theorem~\ref{thm:mean} in the case of a periodic symbol $\varphi$.

\begin{remark}
	If $\psi$ is univalent and $\varphi(s)=\psi(2^{-s})$, then the same argument gives that
	\[\mathscr{M}_\varphi(w) = \log\left|\frac{1}{\psi^{-1}(w)}\right|\]
	for every $w \in \mathbb{C}_{1/2}\setminus \{\varphi(+\infty)\}$. This observation applies to the symbols $\varphi_\nu$ which we discussed in the introduction, and yields a proof of the claim \eqref{eq:Mvarphinu}.
\end{remark}

Inspecting the argument above closely, we can extract two further pieces of information. Firstly, we may actually interchange the limits
\begin{equation} \label{eq:limswap}
	\lim_{\sigma_0\to0^+} \lim_{T\to \infty} \mathscr{M}_\varphi(w,\sigma_0,T) = \lim_{T\to \infty}  \lim_{\sigma_0\to0^+} \mathscr{M}_\varphi(w,\sigma_0,T) =  \mathscr{M}_\varphi(w).
\end{equation}
In the general case, whether \eqref{eq:limswap} holds is an interesting open problem which we feel deserves further work; see Section~\ref{subsec:limswap}. Secondly, we actually have the uniform estimate
\begin{equation} \label{eq:periodicuniform}
	\mathscr{M}_\varphi(w,\sigma_0,T) \leq \mathscr{M}_\varphi(w)\left(1+O(T^{-1})\right) \leq C\log\left|\frac{\overline{w}+\varphi(+\infty)-1}{w-\varphi(+\infty)}\right|.
\end{equation}
for $T\geq1$. This allows us to easily prove Theorem~\ref{thm:cov} in the periodic case, as it can be used to justify taking the limit $T\to\infty$ in the change of variable formula \eqref{eq:covT}, see Section~\ref{sec:compact} below.

For general symbols $\varphi$, we are not able to establish an estimate quite as good as \eqref{eq:periodicuniform}. The rest of this section is devoted to proving some different estimates, uniform in $T$, that are sufficient to prove Theorem~\ref{thm:cov} once the existence of $\mathscr{M}_\varphi$ has been established in Section~\ref{sec:mean}. We begin by clarifying the behavior of the bound in Theorem~\ref{thm:mean} for points $w$ which are far from $\varphi(+\infty)$.

\begin{lemma} \label{lem:logest}
	Let $w$ and $\nu$ be distinct points in $\mathbb{C}_{1/2}$. Then
	\[2\frac{(\mre{w}-1/2)(\mre{\nu}-1/2)}{|\overline{w}+\nu-1|^2} \leq \log\left|\frac{\overline{w}+\nu-1}{w-\nu}\right| \leq 2\frac{(\mre{w}-1/2)(\mre{\nu}-1/2)}{|w-\nu|^2}.\]
\end{lemma}
\begin{proof}
	Apply the elementary inequalities
	\[\frac{1}{2}(1-x^2)\leq -\log{x} \leq \frac{1}{2}(x^{-2}-1), \qquad 0 <x<\infty,\]
	to $x = \big|\frac{w-\nu}{\overline{w}+\nu-1}\big|$.
\end{proof}

We shall now study the mean counting function in the general case via the ordinary Nevanlinna counting function \eqref{eq:nevanlinnadisc}. The approach is insufficient to prove that $\mathscr{M}_\varphi$ exists, but we are able to extract the desired technical estimates. Let $\Theta$ denote the unique conformal map from $\mathbb{D}$ onto the half-strip
\[\sqsubset = \{s = \sigma+it \,:\, \sigma>0, -1<t<1\}\]
with $\Theta(0)=1$ and $\Theta'(0)>0$. A computation reveals that
\[\Theta^{-1}(s) = \frac{\sinh(s \pi/2)-\sinh(\pi/2)}{\sinh(s\pi/2)+\sinh(\pi/2)}.\]
By standard regularity results for conformal maps, there is an absolute constant $C>0$ such that
\begin{equation} \label{eq:Reslog}
	\pi \mre{s} \leq C \log{\frac{1}{|\Theta^{-1}(s)|}}
\end{equation}
whenever $|\mim{s}| \leq 1/2$ and $0\leq \mre{s} \leq 1/2$. For $T > 0$, set
\[\sqsubset_T = \{s=\sigma+it \,:\, \sigma>0, -T<t<T\},\]
and define $\Theta_T \colon \mathbb{D}\to \sqsubset_T$ by $\Theta_T(z) = T \Theta(z)$. Note that $\Theta_T(0)=T$.

Let $\varphi\in \mathscr{G}_0$ and fix some $w \in \mathbb{C}_{1/2}\setminus\{\varphi(+\infty)\}$. Consider the function $\psi_T \colon \mathbb{D}\to \mathbb{D}$ defined by 
\[\psi_T = \Phi_w \circ \Theta_T\]
where $\Phi_w$ is as in \eqref{eq:Phiw}. Note that $\psi_T(0)=\Phi_w(T)$ and that the Nevanlinna counting function of $\psi_T$ at $\xi=0$ satisfies
\begin{equation} \label{eq:psiT}
	N_{\psi_T}(0) = \sum_{z \in \psi_T^{-1}(\{0\})} \log{\frac{1}{|z|}} = \sum_{\substack{s \in \varphi^{-1}(\{w\}) \\ |\mim{s}| < T}} \log{\frac{1}{|\Theta^{-1}_T(s)|}}.
\end{equation}
Based on this we have the following result.

\begin{lemma} \label{lem:ThetaLsub}
	Let $\varphi\in\mathscr{G}_0$ and let $\Phi_w$ be defined as in \eqref{eq:Phiw} for $w \in \mathbb{C}_{1/2}\setminus\{\varphi(+\infty)\}$. There is an absolute constant $C > 0$ such that
	\begin{align}
		\frac{\pi}{T} \sum_{\substack{s \in \Phi_w^{-1}(\{0\}) \\ \,|\mim{s}|< T \\ 0<\mre s < \sigma_1}} \mre{s} &\leq C \log \left|\frac{\overline{w}+ \varphi(2T)-1}{w - \varphi(2T)} \right|, \qquad T \geq \sigma_1 > 0. \label{eq:ThetaLsub1} 
		\intertext{For every $\delta>0$ there are constants $\sigma_2 = \sigma_2(\varphi,\delta)$ and $D = D(\varphi,\delta)$ such that}
		\frac{\pi}{T} \sum_{\substack{s \in \Phi_w^{-1}(\{0\}) \\ |\mim{s}|< T}} \mre{s} &\leq D \log \left| \frac{\overline{w}+\varphi(+\infty)-1}{w - \varphi(+\infty)} \right| \label{eq:ThetaLsub2}	
	\end{align}
	uniformly for $T\geq \sigma_2$ and $|w-\varphi(+\infty)|\geq \delta$.
\end{lemma}

\begin{proof}
	We begin by using that $T\geq \sigma_1$ and \eqref{eq:Reslog} to find that
\begin{align*}	
\frac{\pi}{T} \sum_{\substack{s \in \Phi_w^{-1}(\{0\}) \\ \,|\mim{s}|< T \\ 0<\mre s < \sigma_1}} \mre{s} &\leq C\sum_{\substack{s \in \Phi_w^{-1}(\{0\}) \\ \,|\mim{s}|< T \\ 0<\mre s < \sigma_1}} \log{\frac{1}{|\Theta_{2T}^{-1}(s)|}} \\
&\leq C\sum_{\substack{s \in \Phi_w^{-1}(\{0\}) \\ |\mim{s}|< 2T}} \log{\frac{1}{|\Theta_{2T}^{-1}(s)|}} = CN_{\psi_{2T}}(0),
\end{align*}
	where $N_{\psi_{2T}}$ is from \eqref{eq:psiT}. By the Littlewood inequality \eqref{eq:littlewoodineq} we obtain \eqref{eq:ThetaLsub1}. 

	To establish \eqref{eq:ThetaLsub2}, we choose $\sigma_2\geq1$ so large that $|\varphi(+\infty)-\varphi(s)|\leq \delta/2$ for $s$ in the half-plane $\mathbb{C}_{\sigma_2}$. In particular, for $s \in \mathbb{C}_{\sigma_2}$
	\[|w-\varphi(s)| \geq |w-\varphi(+\infty)|-|\varphi(+\infty)-\varphi(s)|\geq \delta/2.\]
	Therefore, the equation $\Phi_w(s)=0$ has no solutions in $\overline{\mathbb{C}_{\sigma_2}}$. By \eqref{eq:ThetaLsub1}, with $\sigma_1=\sigma_2$, we have that
	\[\frac{\pi}{T} \sum_{\substack{s \in \Phi_w^{-1}(\{0\}) \\ |\mim{s}|< T}} \mre{s} = \frac{\pi}{T} \sum_{\substack{s \in \Phi_w^{-1}(\{0\}) \\ \,|\mim{s}|< T \\ 0<\mre s < \sigma_2}} \mre{s} \leq C \log \left|\frac{\overline{w}+ \varphi(2T)-1}{w - \varphi(2T)} \right|\]
	for $T\geq \sigma_2$. To estimate the right hand side, we want to use the upper bound in Lemma~\ref{lem:logest} with $\nu=\varphi(2T)$ and the lower bound in Lemma~\ref{lem:logest} with $\nu=\varphi(+\infty)$. We then obtain \eqref{eq:ThetaLsub2} with
	\[D(\varphi,\delta) = \sup{\frac{\mre{\varphi(2T)}-1/2}{\mre{\varphi(+\infty)}-1/2} \frac{|w+\varphi(+\infty)-1|^2}{|w-\varphi(2T)|^2}} < \infty,\]
where the supremum is taken over all $T \geq \sigma_2$ and $|w-\varphi(+\infty)| \geq \delta$.
\end{proof}

\section{$\mathscr{H}^p$-theory I: Convexity} \label{sec:convexity}
The purpose of the present section is to collect some results regarding Hardy spaces of Dirichlet series from \cite{Bayart02,BBSS19,CG86,HLS97} and to obtain some new results on the convexity of certain integrals.

Let $\mathscr{P}$ denote the set of all Dirichlet polynomials $f(s)=\sum_{n\leq N} a_n n^{-s}$. For $0<p<\infty$, we define the Hardy space $\mathscr{H}^p$ as the closure of $\mathscr{P}$ with respect to
\begin{equation} \label{eq:curlyHp}
	\|f\|_{\mathscr{H}^p}^p = \lim_{T\to\infty} \frac{1}{2T}\int_{-T}^T |f(it)|^p\,dt.
\end{equation}
Since Dirichlet polynomials are almost periodic, it is clear that the limit \eqref{eq:curlyHp} exists. Moreover, $\|f\|_{\mathscr{H}^p}$ is a norm when $1\leq p <\infty$ and a quasi-norm when $0<p<1$. In the latter case, note that $\|f-g\|_{\mathscr{H}^p}^p$ defines a metric on $\mathscr{P}$.

To make sense of this definition, we recall that the \emph{Bohr correspondence} allows us to identify $\mathscr{H}^p$ with certain Hardy spaces of the (countably) infinite polydisc 
\[\mathbb{D}^\infty=\mathbb{D}\times\mathbb{D}\times\cdots.\]
Let $\operatorname{Poly}_\infty$ denote the set of all analytic polynomials in an arbitrary number of variables and let $\mathscr{B}\colon \mathscr{P}\to\operatorname{Poly}_\infty$ denote the multiplicative bijection defined by $F = \mathscr{B}f$ where
\[f(s) = \sum_{n=1}^N a_n n^{-s} \qquad \text{and} \qquad F(\chi) = \sum_{n=1}^N a_n \chi(n).\]
For $0<p<\infty$, define $H^p(\mathbb{D}^\infty)$ as the closure of $\operatorname{Poly}_\infty$ in the norm of $L^p(\mathbb{T}^\infty)$. By results of \cite{Bayart02,BBSS19,HLS97}, we know that $\mathscr{B}$ extends to a multiplicative isometric isomorphism from $\mathscr{H}^p$ to $H^p(\mathbb{D}^\infty)$.

It is demonstrated in \cite{CG86} that via Poisson extension, $H^p(\mathbb{D}^\infty)$ can be seen as a space of convergent power series in $\mathbb{D}^\infty \cap \ell^2$. In particular, this implies that $\mathscr{H}^p$ is a space of absolutely convergent Dirichlet series in the half-plane $\mathbb{C}_{1/2}$. As in the case $p=2$, if $f\in\mathscr{H}^p$, then $\sigma_{\operatorname{c}}(f_\chi)\leq0$ for almost every $\chi \in \mathbb{T}^\infty$. Moreover, the generalized boundary value \eqref{eq:genboundary} exists for almost every $\chi\in\mathbb{T}^\infty$. For these claims we refer to \cite[Thm.~5]{Bayart02}, and note that the proofs given there carry over to the range $0<p<1$ once the results of \cite[Sec.~2]{BBSS19} are known. 

Since we will swap between the polydisc and half-plane points of view, it will be convenient to consistently use the notation
\begin{equation} \label{eq:Ffsigma}
	F(\chi) = f^\ast(\chi) \qquad \text{and} \qquad F_\sigma(\chi) = f_\chi(\sigma),
\end{equation}
where $\sigma \geq 0$. 
 Note that the identifications \eqref{eq:Ffsigma} makes sense for almost every $\chi \in \mathbb{T}^\infty$, whenever $f \in \mathscr{H}^p$ for some $0<p<\infty$.

For a positive integer $m$, let $A_m$ denote what Bohr called \emph{der} $m$\emph{te Abschnitt}, defined formally by
\[A_mF(\chi)=F(\chi_1,\chi_2,\ldots,\chi_m,0,0,\ldots).\]
Note that if $F(\chi)=\sum_{n\geq1} a_n \chi(n)$ then, $A_m F$ is obtained by replacing $a_n$ by $0$ whenever $p_j|n$ for some $j>m$. Der $m$te Abschnitt is often a useful tool when we desire to extend results from the finite polydisc $\mathbb{D}^m$ to the infinite polydisc $\mathbb{D}^\infty$. Recall in particular \cite[Thm.~2.1]{BBSS19}, which states that $F\in H^p(\mathbb{D}^\infty)$ if and only if 
\[\sup_{m\geq1} \|A_m F\|_{H^p(\mathbb{D}^m)}<\infty\]
and moreover that if $F\in H^p(\mathbb{D}^\infty)$, then $A_m F \to F$ in the norm of $H^p(\mathbb{D}^\infty)$. The following result is a typical application of der $m$te Abschnitt. We omit the proof, which is a combination of the corresponding result for $H^p(\mathbb{D}^m)$, essentially found in \cite[Sec.~3.4]{Rudin69}, and \cite[Thm.~2.1]{BBSS19}.

\begin{lemma} \label{lem:Hpintdown}
	Fix $0<p<\infty$. Suppose that $F(\chi)$ is given by a convergent power series in $\mathbb{D}^\infty \cap \ell^2$. Then $F\in H^p(\mathbb{D}^\infty)$ if and only if 
	\[\sup_{\sigma>0} \|F_\sigma\|_{H^p(\mathbb{D}^\infty)}<\infty.\] 
	Moreover, if $F\in H^p(\mathbb{D}^\infty)$, then $\|F_\sigma-F\|_{H^p(\mathbb{D}^\infty)}\to 0$ as $\sigma\to 0^+$.
\end{lemma}

From Lemma~\ref{lem:Hpintdown} and \eqref{eq:Ffsigma}, we get the formula
\begin{equation} \label{eq:carlsonchi}
	\|f\|_{\mathscr{H}^p}^p = \lim_{\sigma\to0^+} \int_{\mathbb{T}^\infty} |f_\chi(\sigma)|^p\,dm_\infty(\chi).
\end{equation}
From \eqref{eq:carlsonchi}, we can obtain the following analogue of Carlson's formula \eqref{eq:carlson2} for $\mathscr{H}^p$. Namely, if $f$ is a Dirichlet series with $\sigma_{\operatorname{u}}(f) \leq 0$, then
\begin{equation} \label{eq:carlsonp}
	\|f\|_{\mathscr{H}^p}^p = \lim_{\sigma\to0^+} \lim_{T\to\infty} \frac{1}{2T} \int_{-T}^T |f(\sigma+it)|^p \,dt.	
\end{equation}
To obtain \eqref{eq:carlsonp} from \eqref{eq:carlsonchi}, one can either use Birkhoff's ergodic theorem (see the proof of Lemma~\ref{lem:Jessen} below) or a more elementary argument involving the Weierstrass approximation theorem (see~\cite[Sec.~3]{SS09}). 

The following result illustrates that $\mathscr{H}^p$-theory is relevant to the study of composition operators on $\mathscr{H}^2$, since it applies to the Gordon--Hedenmalm class $\mathscr{G}$.

\begin{lemma} \label{lem:map0p1}
	Fix $0<p<1$. Suppose that $\varphi$ is a Dirichlet series with $\sigma_{\operatorname{u}}(\varphi)\leq0$ and that $\varphi(\mathbb{C}_0)\subseteq \mathbb{C}_0$. Then $\varphi\in\mathscr{H}^p$ and 
	\begin{equation} \label{eq:kolmogorov}
		\|\varphi\|_{\mathscr{H}^p}^p \leq \cos\left(\frac{p \pi}{2}\right)^{-1} \big(\mre{\varphi(+\infty)}\big)^p + \big|\mim{\varphi(+\infty)}\big|^p.
	\end{equation}
\end{lemma}

\begin{proof}
	Suppose that $\varphi(s)=\sum_{n\geq1} c_n n^{-s}$, so that $\varphi(+\infty)=c_1 \in \mathbb{C}_0$. We begin by assuming that $c_1>0$. Due to the assumption $\varphi$ does not vanish in $\mathbb{C}_0$, we can define the analytic function $\varphi_p = \varphi^p$ in $\mathbb{C}_0$. Since $\sigma_{\operatorname{u}}(f)\leq 0$, we know that $\varphi$ is bounded on every half-plane $\mathbb{C}_\varepsilon$ for $\varepsilon>0$, and hence the same holds for $\varphi_p$. Moreover, since $\varphi$ converges absolutely in $\mathbb{C}_{1/2}$ by \eqref{eq:au12}, we may expand
	\[\varphi_p(s) = c_1^p \left(1 + c_1^{-p} \sum_{n=2}^\infty c_n n^{-s}\right)^p\]
	as a Dirichlet series for every $s$ with sufficiently large real part. Lemma~\ref{lem:bohr} therefore yields that $\sigma_{\operatorname{u}}(\varphi_p)\leq0$. The uniform convergence implies that for every $\sigma>0$ we have the mean value
	\[\lim_{T\to\infty} \frac{1}{2T} \int_{-T}^T \varphi_p(\sigma+it)\,dt = c_1^p\]
	By definition of $\varphi_p$, we have that $|\Arg{\varphi_p(s)}|\leq p \pi/2$. Taking real parts, we obtain
	\[c_1^p = \lim_{T\to\infty} \frac{1}{2T} \int_{-T}^T \mre{\varphi_p(\sigma+it)}\,dt \geq \lim_{T\to\infty} \frac{1}{2T} \int_{-T}^T |\varphi(\sigma+it)|^p \cos\left(\frac{p \pi}{2}\right)\,dt\]
	for every $\sigma>0$. Letting $\sigma\to0^+$ and using \eqref{eq:carlsonp}, we conclude $\varphi\in\mathscr{H}^p$ and that
	\[\|\varphi\|_{\mathscr{H}^p}^p \leq \cos\left(\frac{p \pi}{2}\right)^{-1} c_1^p,\]
	under the assumption that $\mim{c_1}=0$. Writing $\varphi(s) = \varphi(s)-\mim{c_1} + \mim{c_1}$ and using the triangle inequality for $0<p<1$, we obtain \eqref{eq:kolmogorov}.
\end{proof}

\begin{remark}
The symbols $\varphi_\nu$ from \eqref{eq:varphinu} satisfy the hypotheses of Lemma~\ref{lem:map0p1}, and hence $\varphi_\nu \in \mathscr{H}^p$ for every $0<p<1$, but $\varphi_\nu \notin \mathscr{H}^1$.
\end{remark}

From Lemma~\ref{lem:map0p1} and the fact that elements of $\mathscr{H}^p$ have generalized boundary values almost everywhere on $\mathbb{T}^\infty$, we obtain the following corollary, mentioned near the statement of Theorem~\ref{thm:sharpiro} in the introduction. A different proof of the same result can be found in \cite[Sec.~2]{BP20}.

\begin{corollary} \label{cor:varphiboundary}
	Suppose that $\varphi$ is a Dirichlet series with $\sigma_{\operatorname{u}}(\varphi)\leq0$ and that $\varphi(\mathbb{C}_0)\subseteq \mathbb{C}_0$. Then the generalized boundary value
	\[\varphi^\ast(\chi) = \lim_{\sigma\to0^+} \varphi_\chi(\sigma)\]
	exists for almost every $\chi \in \mathbb{T}^\infty$.
\end{corollary}

Our next goal is to investigate integrals involving $\log|F_\sigma|$, for $F \in H^p(\mathbb{D}^\infty)$. In particular, we want to prove that such integrals are convex non-increasing functions of $\sigma$. In the case of $H^p(\mathbb{D})$, such results are typically obtained as consequences of subharmonicity (see e.g.~\cite[Sec.~1.3--1.4]{Duren70}). However, iteration of one-variable convexity does not seem to yield the result that we desire. In our next two results, we will therefore instead adapt an argument due to Hardy \cite{Hardy15}.

\begin{lemma} \label{lem:hardy}
	Suppose that $g$ is analytic in $\mathbb{C}_0$ and that $\psi \colon \mathbb{R} \to \mathbb{R}$ is a $C^2$ function such that $\operatorname{supp}(\psi') \cap (-\infty, 0]$ is compact. For every $s \in \mathbb{C}_0$,
	\[\Delta\, \psi(\log{|g(s)|}) = \left|\frac{g'(s)}{g(s)} \right|^2\psi''(\log |g(s)|)\]
	where $\Delta = \frac{\partial^2}{\partial \sigma^2} + \frac{\partial^2}{\partial t^2}$.
\end{lemma}

\begin{proof}
	The idea is to factor $\Delta = \partial_{\overline{s}} \partial_s$, where $\partial_{\overline{s}} = \frac{\partial}{\partial \sigma} + i \frac{\partial}{\partial t}$ and $\partial_s = \frac{\partial}{\partial \sigma} - i \frac{\partial}{\partial t}$. By the chain rule and the Cauchy--Riemann equations, we get
	\[\partial_s  \psi(\log{|g(s)|}) = \psi'(\log{|g(s)|}) \, \partial_s  \log|g(s)| 	= \psi'(\log{|g(s)|}) \,\frac{g'(s)}{g(s)},\]
	since $\log{|g(s)|}=\mre{\log{g(s)}}$. Note that by the assumptions on $\psi$, both sides of this equation are identically $0$ near any point $s\in\mathbb{C}_0$ where $g(s)=0$. By the product rule and the chain rule, we then have
	\[\Delta\, \psi(\log{|g(s)|}) = \partial_{\overline{s}} \left(\psi'(\log{|g(s)|}) \,\frac{g'(s)}{g(s)}\right) = \psi''(\log|g(s)|) \, \left|\frac{g'(s)}{g(s)} \right|^2\]
	where we in the final equality used that $\partial_{\overline{s}} \log|g(s)|=\overline{\partial_{s} \log|g(s)|}$ and the Cauchy--Riemann equations. 
\end{proof}

\begin{lemma} \label{lem:convex}
	Fix $F\in H^p(\mathbb{D}^\infty)$ for some $0<p<\infty$. Let $\psi \colon \mathbb{R} \to \mathbb{R}$ be a $C^2$ function such that $\operatorname{supp}(\psi') \cap (-\infty, 0]$ is compact and let $m\geq1$. For $\sigma>0$, set
	\[\Psi_m(\sigma) = \int_{\mathbb{T}^\infty} \psi(\log{|A_m F_\sigma(\chi)|})\,dm_\infty(\chi).\]
	If $\psi$ is convex, then $\Psi_m$ is convex and if $\log\circ\, \psi$ is convex, then $\log \circ\,\Psi_m$ is convex.
\end{lemma}

\begin{proof}
	Recalling the notation of \eqref{eq:Ffsigma}, we see that $g(s)=A_m f_\chi(s)$ is absolutely convergent in $\mathbb{C}_0$ for any fixed $m\geq1$ and every $\chi\in\mathbb{T}^\infty$. Combining this with the assumptions on $\psi$, the integral
	\[\Psi_m(s) = \int_{\mathbb{T}^\infty} \psi(\log{|g(s)|})\,dm_\infty(\chi).\]
	exists for every $s \in \mathbb{C}_0$. Moreover, since $dm_\infty$ is invariant under rotations in each variable, we actually see that $\Psi_m(s)=\Psi_m(\sigma)$ where $\sigma=\mre{s}$. Applying Lemma~\ref{lem:hardy} to $\psi(\log|g(s)|)$ and integrating over $\chi$, we find that
	\[\Delta\, \Psi_m(s) = \int_{\mathbb{T}^\infty} \Delta\, \psi(\log|g(s)|)\,dm_\infty(\chi) = \int_{\mathbb{T}^\infty} \psi''(\log|g(s)|)\left|\frac{g'(s)}{g(s)}\right|^2\,dm_\infty(\chi).\]
	However, recalling that $\Psi_m(\sigma+it)$ is constant as a function of $t$, we find that $\Delta\, \Psi_m(s) = \frac{\partial^2}{\partial \sigma^2 }\Psi_m(\sigma)$. Hence we have found that
	\begin{equation} \label{eq:convex1}
		\frac{\partial^2}{\partial \sigma^2 } \Psi_m(\sigma) = \int_{\mathbb{T}^\infty} \psi''(\log{|A_m F_\sigma(\chi)|})\left|\frac{\frac{\partial}{\partial \sigma} A_m F_\sigma(\chi)}{A_m F_\sigma(\chi)}\right|^2\,dm_\infty(\chi).
	\end{equation}
	If the $C^2$ function $\psi$ is convex, then $\psi''\geq0$, and we see from \eqref{eq:convex1} that $\Psi_m$ is convex. 

	To prove the second statement, we again use that $\Psi_m(s)=\Psi_m(\sigma)$ so that, with $\partial_s$ as in the proof of Lemma~\ref{lem:hardy}, 
	\begin{equation} \label{eq:convex2}
		\frac{\partial}{\partial \sigma} \Psi_m(\sigma) = \partial_s \Psi_m(s) = \int_{\mathbb{T}^\infty} \psi'(\log{|A_m F_\sigma(\chi)|})\,\frac{\frac{\partial}{\partial \sigma} A_m F_\sigma(\chi)}{A_m F_\sigma(\chi)}\,dm_\infty(\chi).
	\end{equation}
	If the $C^2$ function $\psi$ is log-convex, then  $(\psi')^2 \leq \psi \psi''$. 
By applying the Cauchy--Schwarz inequality on \eqref{eq:convex2} and using \eqref{eq:convex1} we see that also $(\Psi'_m)^2 \leq \Psi_m \Psi''_m$, so that $\Psi_m$ is log-convex.
\end{proof}

We are now in a position to obtain our result for the integrals of $\log{|F_\sigma|}$. We begin with the following basic result.

\begin{lemma} \label{lem:Nout}
	Suppose that $F\in H^p(\mathbb{D}^\infty)$ for some $0<p<\infty$, and that $F\not\equiv0$. Then there is a positive integer $N$ and a constant $a_N\neq0$ such that
	\[\int_{\mathbb{T}^\infty} \log{|F_\sigma(\chi)|}\,dm_\infty(\chi) = -\sigma \log{N} + \log{|a_N|}\]
	for all sufficiently large $\sigma>0$.
\end{lemma}

\begin{proof}
	Since $F\not\equiv 0$, there is some integer $N$ such that $F(\chi) =\sum_{n\geq N} a_n \chi(n)$ and $a_N \neq 0$. We write
	\[F_\sigma(\chi) = N^{-\sigma} a_N \chi(N) \big(1 + \overline{\chi(N)}a_N^{-1} G_\sigma(\chi)\big), \qquad G_\sigma(\chi) = \sum_{n=N+1}^\infty \left(\frac{N}{n}\right)^\sigma a_n \chi(n).\]
	Since $F\in H^p(\mathbb{D}^\infty)$, there is a $\sigma_0 > 0$ such that if $\sigma\geq \sigma_0$, then $G_\sigma$ converges absolutely for $\chi\in\mathbb{T}^\infty$ and satisfies $|G_\sigma(\chi)|\leq |a_N|/2$. Hence we may expand the logarithm to obtain that
	\[\log\big(1 + \overline{\chi(N)}a_N^{-1} G_\sigma(\chi)\big) = \sum_{k=1}^\infty \frac{(-1)^{k-1}}{k}\big(\overline{\chi(N)}a_N^{-1} G_\sigma(\chi)\big)^k.\]
	However, for every $k\geq1$
	\[\int_{\mathbb{T}^\infty}\big(\overline{\chi(N)}a_N^{-1} G_\sigma(\chi)\big)^k\,dm_\infty(\chi)=0.\]
	To see this, note that $\chi(N)^k = \chi(N^k)$, by the fact that $\chi$ is completely multiplicative, while the expansion of $G_\sigma^k$ only contains characters $\chi(n)$ with $n\geq(N+1)^k$. The proof is completed by using  
	\[\int_{\mathbb{T}^\infty} \log\big|1 + \overline{\chi(N)}a_N^{-1} G_\sigma(\chi)\big|\,dm_\infty(\chi) =\mre{\int_{\mathbb{T}^\infty} \log\big(1 + \overline{\chi(N)}a_N^{-1} G_\sigma(\chi)\big)\,dm_\infty(\chi)},\]
	the fact that $|\chi(N)|=1$ and the relationship between $F_\sigma$ and $G_\sigma$.
\end{proof}

We require some notation. For $x \geq 0$, set $\log^+{x}=\max(\log{x},0)$ and $\log^-{x} = \log^+{x}-\log{x}$. We will have use of the elementary inequality 
\begin{equation} \label{eq:log+ineq}
	\left|\log^+{x} - \log^+{y}\right| \leq \frac{|x-y|^p}{p},
\end{equation} 
valid for $x,y \geq 0$ and $0 < p \leq 1$.

\begin{theorem} \label{thm:logintdown}
	Suppose that $F\in H^p(\mathbb{D}^\infty)$ for some $0<p<\infty$, and that $F\not\equiv0$. Then $\log|F_\sigma|\in L^1(\mathbb{T}^\infty)$ for every $\sigma\geq0$ and
	\begin{equation} \label{eq:logintdown}
		\int_{\mathbb{T}^\infty} \log{|F_\sigma(\chi)|}\,dm_\infty(\chi)
	\end{equation}
	is a convex non-increasing function of $\sigma \geq 0$.
\end{theorem}

\begin{proof}
	Fix some $F\not \equiv 0$ in $H^p(\mathbb{D}^\infty)$. There is some $m_0$ such that if $m\geq m_0$, then $A_m F \not \equiv 0$. For a given $-\infty < \beta < 0$, choose a sequence $(\psi_j)_{j\geq1}$ of convex $C^2$ functions such that $\operatorname{supp}(\psi_j') \cap (-\infty, 0]$ is compact, $\beta - 1 \leq \psi_j(x) \leq \max(x,\beta+1)$ for all $x$, and $\psi_j(x) \to \max(x, \beta)$ point-wise as $j\to\infty$. For $m\geq m_0$, we define
	\[\Psi_{m,j}(\sigma) = \int_{\mathbb{T}^\infty} \psi_j(\log|A_m F_\sigma(\chi)|)\,dm_\infty(\chi).\]
	By Lemma~\ref{lem:convex}, we see that $\Psi_{m,j}$ is a convex function of $\sigma > 0$. For $-\infty < \beta < 0$, set $\log_{\beta} x = \max(\log x, \beta)$. By taking $j\to\infty$, we conclude that
	\begin{equation} \label{eq:betam}
		\int_{\mathbb{T}^\infty} \log_\beta{|A_m F_\sigma(\chi)|}\,dm_\infty(\chi).
	\end{equation}
	is convex function of $\sigma>0$. Note that by \eqref{eq:log+ineq}, 
	\[\left|\log_\beta |A_m F_{\sigma}(\chi)|\right| \leq \max\left(\log^+ |A_m F_{\sigma}(\chi)|,\, |\beta|\right) \leq \frac{1}{p} \max\left( |A_m F_{\sigma}(\chi)|^p,\, p|\beta|\right).\]
	Recall that $F_\sigma \in H^p(\mathbb{D}^\infty)$ for every $\sigma>0$ by Lemma~\ref{lem:Hpintdown}, and that $A_m F_\sigma \to F_\sigma$ in $H^p(\mathbb{D}^\infty)$ as $m \to \infty$, by \cite[Thm.~2.1]{BBSS19}. This implies that $\log_\beta |A_m F_{\sigma}|$ is uniformly integrable in $m\geq1$, for any $\sigma > 0$. Therefore $\log_\beta |F_{\sigma}|\in L^1(\mathbb{T}^\infty)$ for $\sigma > 0$, and taking $m \to \infty$ in \eqref{eq:betam} yields that
	\begin{equation} \label{eq:betaconvex1}
		\int_{\mathbb{T}^\infty} \log_\beta{|F_\sigma(\chi)|}\,dm_\infty(\chi)
	\end{equation}
	is a convex function of $\sigma>0$. We now let $\beta \to -\infty$. Since \eqref{eq:logintdown} is finite for all sufficiently large $\sigma$ by Lemma~\ref{lem:Nout}, the monotone convergence theorem and the convexity of \eqref{eq:betaconvex1} implies that $\log {|F_\sigma|} \in  L^1(\mathbb{T}^\infty)$ for all $\sigma > 0$, and that \eqref{eq:logintdown} is a convex function of $\sigma>0$. 
	
	Next, let us prove that \eqref{eq:logintdown} is non-increasing. By Lemma~\ref{lem:Nout} we see that this holds for all sufficiently large $\sigma$. However, by convexity, we know that \eqref{eq:logintdown} is continuous and has a non-decreasing right-derivative at every point, so it must be non-increasing as function of $\sigma$ for every $\sigma>0$.
	 
	Finally, we consider the case $\sigma = 0$, with $F_0 = F$. By Lemma~\ref{lem:Hpintdown} and \eqref{eq:log+ineq} we immediately see that $\log^+ |F_\sigma| \to \log^+ |F|$ in $L^1(\mathbb{T}^\infty)$ as $\sigma \to 0^+$. By the the fact that \eqref{eq:logintdown} is non-increasing and finite for $\sigma > 0$, we find that
	\[\int_{\mathbb{T}^\infty} \log^{-}{|F(\chi)|}\,dm_\infty(\chi) \leq \varliminf_{\sigma \to 0^+} \int_{\mathbb{T}^\infty} \log^{-}{|F_\sigma(\chi)|}\,dm_\infty(\chi) < \infty,\]
where we have also applied Fatou's lemma for the first inequality. Therefore $\log |F| \in L^1(\mathbb{T}^\infty)$ and \eqref{eq:logintdown} is upper semi-continuous at $\sigma = 0$,
\[\int_{\mathbb{T}^\infty} \log {|F(\chi)|}\,dm_\infty(\chi) \geq \varlimsup_{\sigma \to 0^+} \int_{\mathbb{T}^\infty} \log {|F_\sigma(\chi)|}\,dm_\infty(\chi),\]
finishing the proof that \eqref{eq:logintdown} is a convex function of $\sigma \geq 0$.
\end{proof}

\begin{remark}
	The fact that $\log |F| \in L^1(\mathbb{T}^\infty)$ whenever $0 \not\equiv F \in H^p(\mathbb{D}^\infty)$ also appeared in \cite{AOS19}, where a subharmonicity argument was employed.
\end{remark}

We conclude this section with the following result, which is an $\mathscr{H}^p$-analogue of Hardy's convexity theorem (c.f.~\cite[Thm.~1.5]{Duren70}). We omit the details of the proof, but note that (i) can easily be deduced iteratively from the corresponding result for $H^p(\mathbb{D})$  and that (ii) has a proof similar to, but much easier than, that of Theorem~\ref{thm:logintdown}, using the second statement of Lemma~\ref{lem:convex}. 

\begin{theorem} \label{thm:hardy}
    Let $0<p<\infty$. Suppose that $F \in H^p(\mathbb{D}^\infty)$, and define
   	\[\mathscr{A}_p(\sigma,F)=\int_{\mathbb{T}^\infty} |F_\sigma(\chi)|^p\,dm_\infty(\chi)\]
	for $\sigma\geq0$. Then
	\begin{enumerate}
		\item[(i)] $\mathscr{A}_p(\sigma,F)$ is a non-increasing function of $\sigma$,
		\item[(ii)] $\log{\mathscr{A}_p(\sigma,F)}$ is a convex function of $\sigma$.
	\end{enumerate}
\end{theorem}

It is interesting to note that the corresponding result for $H^p(\mathbb{D})$ is originally from Hardy's paper \cite{Hardy15}, which is often considered the historical starting point of the theory of $H^p(\mathbb{D})$.

\section{$\mathscr{H}^p$-theory II: Jessen's function and Frostman's theorem} \label{sec:jessen}
One goal of this paper is to illustrate that the \emph{Jessen} function
\begin{equation} \label{eq:Jessen}
	\mathscr{J}_f(\sigma) = \lim_{T\to\infty} \frac{1}{2T}\int_{-T}^T \log{|f(\sigma+it)|}\,dt.
\end{equation}
plays the same role in $\mathscr{H}^p$-theory as the corresponding Jensen function plays in $H^p(\mathbb{D})$-theory (see \cite[Ch.~2]{Duren70} and \cite[Sec.~7.3]{Shapiro93}). Here $f$ is a Dirichlet series with $\sigma_{\operatorname{u}}(f)\leq0$. The existence of $\mathscr{J}_f(\sigma)$ for $\sigma > 0$ is due to Jessen \cite[Satz~A]{Jessen33} (see also \cite[Thm.~5]{JT45}), and is proven by using almost periodicity to control the zeros of $f$ in $\mathbb{C}_{\sigma/2}$. The following result is essentially a corollary of the work of Jessen.
\begin{lemma} \label{lem:Jessen}
	Suppose that $\sigma_{\operatorname{u}}(f)\leq0$. For every $\sigma>0$ the Jessen function \eqref{eq:Jessen} exists and is a convex non-increasing function of $\sigma$. Moreover,
	\begin{equation} \label{eq:Jessenchi}
		\mathscr{J}_f(\sigma) = \int_{\mathbb{T}^\infty} \log{|f_\chi(\sigma)|}\,dm_\infty(\chi).
	\end{equation}
\end{lemma}

\begin{proof}
	Jessen \cite[Satz~A]{Jessen33} proved that if $\sigma_{\operatorname{u}}(f)\leq0$, then the limit \eqref{eq:Jessen} exists for every $\sigma>0$ and that if $(f_j)_{j\geq1}$ converges uniformly to $f$, then $\mathscr{J}_{f_j}(\sigma)\to \mathscr{J}_f(\sigma)$ in the half-plane of uniform convergence. Combining this with \eqref{eq:chitauk} and the fact that \eqref{eq:Jessen} is invariant under vertical translations, i.e. that the Jessen functions of $f$ and $T_\tau f$ are identical, we conclude that $\mathscr{J}_f(\sigma)=\mathscr{J}_{f_\chi}(\sigma)$ for every $\chi \in \mathbb{T}^\infty$.
	
	If $\sigma_{\operatorname{u}}(f)\leq0$ and $\sigma>0$, then $f(\cdot+\sigma)\in\mathscr{H}^p$ for every $0<p<\infty$. We therefore know that $\log|f_\chi(\sigma)|$ belongs to $L^1(\mathbb{T}^\infty)$, by Theorem~\ref{thm:logintdown}. This allows us to apply Birkhoff's ergodic theorem (see e.g. \cite[Thm.~2.1.12]{QQ13}) for the Kronecker flow
	\[\mathscr{T}_t = (2^{-it},3^{-it},\ldots,p_j^{-it},\ldots)\]
	on $\mathbb{T}^\infty$, to conclude that there is at least one $\chi' \in\mathbb{T}^{\infty}$ such that
	\[\int_{\mathbb{T}^\infty} \log|f_\chi(\sigma)|\,dm_\infty(\chi) =\lim_{T\to\infty} \frac{1}{2T}\int_{-T}^T \log|f_{\chi'}(\sigma + it)|\,dt = \mathscr{J}_{f_{\chi'}}(\sigma) = \mathscr{J}_f(\sigma).\]
	Hence we have established that the integrals \eqref{eq:Jessen} and \eqref{eq:Jessenchi} coincide when $\sigma_{\operatorname{u}}(f)\leq0$ and $\sigma>0$. The fact that the resulting function is convex and non-increasing follows either by Jessen's work \cite{Jessen33} or by our Theorem~\ref{thm:logintdown}.
\end{proof}

Note that for general $f \in \mathscr{H}^p$, the Jessen function \eqref{eq:Jessen} does not necessarily exist for $\sigma < 1/2$. In Theorem~\ref{thm:logintdown} we have therefore extended Jessen's work to $\mathscr{H}^p$-functions, by defining the Jessen function by \eqref{eq:Jessenchi} and removing the assumption that $\sigma_{\operatorname{u}}(f)\leq0$. This should be compared with the two versions of Carlson's formula \eqref{eq:carlsonchi} and \eqref{eq:carlsonp}.  Jessen proved the convexity of $\mathscr{J}_f(\sigma)$ through a different argument which relies on the Cauchy--Riemann equations and the argument principle to connect $\mathscr{J}_f'(\sigma)$ with an unweighted mean counting function; see Section~\ref{sec:JT}.

When we use the Jessen function to study the mean counting function in Section~\ref{sec:mean}, it is crucial that the limit of $\mathscr{J}_f(\sigma)$ exists when $\sigma\to0^+$. For this purpose, we introduce the Nevanlinna class $\mathscr{N}_{\operatorname{u}}$ of Dirichlet series $f$ with $\sigma_{\operatorname{u}}(f) \leq 0$ such that
\begin{equation} \label{eq:Nudef}
	\varlimsup_{\sigma \to 0^+}  \lim_{T \to \infty} \frac{1}{2T} \int_{-T}^T \log^+ |f(\sigma + it)| \, dt < \infty.
\end{equation}
Note that if $\sigma_{\operatorname{u}}(f) \leq 0$ and $f \in \mathscr{H}^p$, then $f \in \mathscr{N}_{\operatorname{u}}$, as follows from the inequality $\log^+ |f| \leq (1/p)|f|^p$ from \eqref{eq:log+ineq} and Carlson's formula \eqref{eq:carlsonp}.

\begin{remark}
	Using der $m$te Abschnitt as above, it is possible to define a more general Nevanlinna--type class of Dirichlet series, better suited to $\mathscr{H}^p$-theory. However, the class $\mathscr{N}_{\operatorname{u}}$ is natural for the study of the mean counting function. In either case, the subclass consisting of periodic Dirichlet series $\psi(2^{-s})$ is, under the identification $z=2^{-s}$, the usual Nevanlinna class $N$ of functions holomorphic in the unit disc $\mathbb{D}$.
\end{remark}

The mean integral in \eqref{eq:Nudef} certainly exists for every $\sigma > 0$, since $\log^+$ is continuous on $[0,\infty)$ so that, $t \mapsto \log^+ |f(\sigma + it)|$ is almost periodic. Here is a version of Lemma~\ref{lem:Jessen} for these mean integrals, which is much easier to prove. We will not actually need the result, but include it here for the sake of completeness.

\begin{lemma}
	Suppose that $\sigma_{\operatorname{u}}(f)\leq 0$. For every $\sigma>0$,
	\begin{equation} \label{eq:logplusint}
		\lim_{T\to\infty} \frac{1}{2T}\int_{-T}^T \log^+|f(\sigma+it)|\,dt = \int_{\mathbb{T}^\infty} \log^+|f_\chi(\sigma)|\,dm_\infty(\chi)
	\end{equation}
	and \eqref{eq:logplusint} defines a convex non-increasing function of $\sigma$.
\end{lemma}

\begin{proof}
	We have already seen that the mean integral on the left hand side of \eqref{eq:logplusint} exists, and since $f(\cdot+\sigma)\in\mathscr{H}^p$ for every $\sigma>0$, the existence of the right hand side follows from the proof of Theorem~\ref{thm:logintdown} with $\beta=0$. To see that they are equal, we can again use Birkhoff's ergodic theorem as in the proof of Lemma~\ref{lem:Jessen}, or employ an argument involving the Weierstrass approximation theorem (see~\cite[Sec.~3]{SS09}). That the right hand side of \eqref{eq:logplusint} is convex as a function of $\sigma$ also follows from the proof of Theorem~\ref{thm:logintdown}. The fact that it is non-increasing can be obtained by iterating the corresponding result for the usual Nevanlinna class $N$ (see e.g. \cite[Sec.~2.1]{Duren70}), in a similar fashion to the proof of Theorem~\ref{thm:hardy} (i).
\end{proof}

We need the following result in the proof of Theorem~\ref{thm:mean}.

\begin{lemma} \label{lem:Nu}
	If $f \in \mathscr{N}_{\operatorname{u}}$, then
	\begin{equation} \label{eq:Nudef2}
		\varlimsup_{\sigma \to 0^+}  \lim_{T \to \infty} \frac{1}{2T} \int_{-T}^T \big|\log |f(\sigma + it)|\big| \, dt < \infty.
	\end{equation}
\end{lemma}
\begin{proof}
	Note that
	\[\lim_{T \to \infty} \frac{1}{2T} \int_{-T}^T \big|\log |f(\sigma + it)|\big| \, dt = \lim_{T \to \infty} \frac{1}{T} \int_{-T}^T \log^+ |f(\sigma + it)| \, dt - \mathscr{J}_f(\sigma).\]
	By Lemma~\ref{lem:Jessen}, $\mathscr{J}_f(\sigma)$ is non-increasing in $\sigma$. It is therefore obvious that \eqref{eq:Nudef} implies \eqref{eq:Nudef2}.
\end{proof}

Our next goal is to generalize Rudin's version of the Frostman theorem for $H^p(\mathbb{D}^m)$ to $H^p(\mathbb{D}^\infty)$, by following what is essentially Rudin's original argument, found in \cite{Rudin67} or \cite[Sec.~3.6]{Rudin69}. This result will be crucial in the proof of Theorem~\ref{thm:sharpiro}.

\begin{theorem} \label{thm:rudinfrostman}
	Suppose that $F\in H^p(\mathbb{D}^\infty)$ for some $0<p<\infty$. Then for quasi-every $\alpha\in\mathbb{C}$ we have
	\begin{equation} \label{eq:rudinfrostman}
		\lim_{\sigma\to0^+} \int_{\mathbb{T}^\infty} \log{|F_\sigma(\chi)-\alpha|}\,dm_\infty(\chi) = \int_{\mathbb{T}^\infty} \log{|F(\chi)-\alpha|}\,dm_\infty(\chi).
	\end{equation}
\end{theorem}

\begin{proof}
	By the elementary inequality \eqref{eq:log+ineq} and the fact that $F_\sigma\to F$ in $H^p(\mathbb{D}^\infty)$ from Lemma~\ref{lem:Hpintdown}, we find that 
	\begin{equation} \label{eq:smirnovinfrostpf}	
		\lim_{\sigma\to0^+} \int_{\mathbb{T}^\infty} \log^+{|F_\sigma(\chi)-\alpha|}\,dm_\infty(\chi) = \int_{\mathbb{T}^\infty} \log^+{|F(\chi)-\alpha|}\,dm_\infty(\chi)
	\end{equation}
	for every $\alpha\in\mathbb{C}$. For $\sigma > 0$, let
	\[B_\sigma(\alpha) = \int_{\mathbb{T}^\infty} \log^{-} |F_\sigma(\chi) - \alpha| \, dm_{\infty}(\chi).\]
	By Theorem~\ref{thm:logintdown} and \eqref{eq:smirnovinfrostpf}, we know that the limit $B(\alpha) = \lim_{\sigma \to 0^+} B_\sigma(\alpha)$ exists, and it remains to show that
	\begin{equation} \label{eq:frostman}
		B(\alpha) = \int_{\mathbb{T}^\infty} \log^{-} |F(\chi) - \alpha| \, dm_{\infty}(\chi)
	\end{equation}
	for quasi-every $\alpha\in\mathbb{C}$.  Suppose to the contrary that there is a compact set $K \subset \mathbb{C}$ of positive logarithmic capacity such that \eqref{eq:frostman} is false for every $\alpha \in K$. By the continuity principle of potential theory \cite[Cor.~2.6.6]{AH96}, there is a positive non-zero measure $\mu$ supported in $K$ such that
	\[E(\lambda) =  \int_K \log^{-} | \lambda - \alpha| \, d\mu(\alpha), \qquad \lambda \in \mathbb{C}\]
	defines a continuous and bounded function. By Fatou's lemma, Fubini's theorem, the dominated convergence theorem (exploiting the continuity and boundedness of $E$), and Fubini's theorem again, we have that
	\begin{align*}
		\int_K B(\alpha) \, d\mu(\alpha) &\leq \varliminf_{\sigma \to 0^+} \int_K B_\sigma(\alpha) \, d\mu(\alpha) \\
		&= \varliminf_{\sigma \to 0^+} \int_{\mathbb{T}^\infty}  E(F_\sigma(\chi)) \, dm_\infty(\chi) \\
		&= \int_{\mathbb{T}^\infty}  E(F(\chi)) \, dm_\infty(\chi) \\ 
		&= \int_K \int_{\mathbb{T}^\infty} \log^{-}|F(\chi) - \alpha| \, dm_\infty(\chi) \, d\mu(\alpha) 
	\end{align*}
	Since $B(\alpha) \geq \int_{\mathbb{T}^\infty} \log^{-}|F(\chi) - \alpha| \, dm_\infty(\chi)$ by Fatou's lemma, we conclude that 
	\[B(\alpha) = \int_{\mathbb{T}^\infty} \log^{-}|F(\chi) - \alpha| \, dm_\infty(\chi)\]
	for $\mu$-almost every $\alpha$. This is a contradiction.
\end{proof}

Suppose that $\sigma_{\operatorname{u}}(f)\leq0$ and that $f$ and $F$ are related as in \eqref{eq:Ffsigma}. Then, by Lemma~\ref{lem:Jessen}, the left hand side of \eqref{eq:rudinfrostman} can be understood in terms of the Jessen function of $f-\alpha$ and the right hand side is an integral involving the generalized boundary function $f^\ast-\alpha$. We can use this to obtain a version of the classical Frostman theorem for Dirichlet series.

Let $\mathscr{H}^\infty$ denote the Banach space of bounded analytic functions in $\mathbb{C}_0$ which may be represented by a Dirichlet series in some half-plane. We set
\[\|f\|_{\mathscr{H}^\infty} = \sup_{\mre{s}>0} |f(s)|.\]
By Lemma~\ref{lem:bohr}, $\sigma_{\operatorname{u}}(f)\leq0$ for every $f \in \mathscr{H}^\infty$. Under the Bohr correspondence we have that $\|f\|_{\mathscr{H}^\infty}=\|f^\ast\|_{L^\infty(\mathbb{T}^\infty)}$ and, consequently, we identify $\mathscr{H}^\infty$ with the Hardy space $H^\infty(\mathbb{D}^\infty)$. The latter space consists, via Poisson extension, of convergent power series in $\mathbb{D}^\infty \cap c_0$ (see~\cite{CG86,HLS97}). From this it is also clear that $\mathscr{H}^\infty\subset \mathscr{H}^p$ for every $0<p<\infty$.

\begin{remark}
	Is is easy to see that the statements of Theorem~\ref{thm:hardy} also hold for $\mathscr{H}^\infty$. In this case (i) follows from the maximum principle and (ii) from the Hadamard three-lines theorem.
\end{remark}

For $\alpha \in \mathbb{D}$, let $\phi_\alpha$ be the automorphism of $\mathbb{D}$ given by
\[\phi_\alpha(z) = \frac{\alpha-z}{1-\overline{\alpha}z}.\]
If $f$ is an analytic function with range contained in $\mathbb{D}$, then the \emph{Frostman shifts} of $f$ are the functions $\phi_\alpha \circ f$ for $\alpha \in \mathbb{D}$.

We say that $f \in \mathscr{H}^\infty$ is \emph{inner} if $|f^\ast(\chi)|=1$ for almost every $\chi \in \mathbb{T}^\infty$. The following result illustrates how Theorem~\ref{thm:rudinfrostman} really is a generalization of the classical Frostman theorem.

\begin{theorem} \label{thm:sigmainner}
	Suppose that $f\in\mathscr{H}^\infty$ is inner. Then the Dirichlet series $\phi_\alpha \circ f$ is inner for every $\alpha\in\mathbb{D}$. Moreover, 
	\[\lim_{\sigma\to0^+} \lim_{T\to\infty} \frac{1}{2T}\int_{-T}^T \log|\phi_\alpha \circ f(\sigma+it)|\,dt = 0\]
	for quasi-every $\alpha \in \mathbb{D}$.
\end{theorem}

\begin{proof}
	The fact that $\phi_\alpha \circ f \in \mathscr{H}^\infty$ follows at once from Lemma~\ref{lem:bohr}. Since $\phi_\alpha(\mathbb{T})=\mathbb{T}$ we also see that $|f^\ast(\chi)|=1$ if and only if $|(\phi_\alpha \circ f)^\ast(\chi)|=1$. Since the Dirichlet series $\phi_\alpha \circ f$, $f-\alpha$ and $1-\overline{\alpha} f$ are in $\mathscr{H}^\infty$, we can apply Lemma~\ref{lem:Jessen} to each of them. By Theorem~\ref{thm:rudinfrostman}, we therefore need to prove that
	\[\lim_{\sigma\to 0^+} \lim_{T\to\infty} \frac{1}{2T} \int_{-T}^T \log{|1-\overline{\alpha}f(\sigma+it)|}\,dt = \int_{\mathbb{T}^\infty} \log{|f^\ast(\chi)-\alpha|}\,dm_\infty(\chi).\]
	Since $f$ is inner,
	\[\int_{\mathbb{T}^\infty} \log{|f^\ast(\chi)-\alpha|}\,dm_\infty(\chi) = \int_{\mathbb{T}^\infty} \log{|1-\overline{\alpha} f^\ast(\chi)|}\,dm_\infty(\chi).\]
	However, for every $\sigma>0$ we have
	\begin{align*}
		\lim_{T\to\infty} \frac{1}{2T} \int_{-T}^T \log{|1-\overline{\alpha}f(\sigma+it)|}\,dt &= \mre{\lim_{T\to\infty} \frac{1}{2T} \int_{-T}^T \log{\big(1-\overline{\alpha}f(\sigma+it)\big)}\,dt} \\
		&= \log|1 - \overline{\alpha} f(+\infty)|
	\end{align*}
	by analyticity and almost periodicity. Similarly, since $\log(1-\overline{\alpha} f^\ast)$ is in $H^\infty(\mathbb{D}^\infty)$ for $|\alpha|<1$, we have that
	\begin{align*}
		\int_{\mathbb{T}^\infty} \log{|1 - \overline{\alpha} f^\ast(\chi)|} \, dm_{\infty}(\chi) &= \mre{\int_{\mathbb{T}^\infty} \log{\big(1 - \overline{\alpha} f^\ast(\chi)\big)} \, dm_{\infty}(\chi)} \\
		&= \log|1 - \overline{\alpha} f^\ast(0)|,
	\end{align*}
	where $f^\ast(0)$ denotes the value of the Poisson extension of $f^\ast$ at $0$. But of course, $f^\ast(0) = f(+\infty)$, and so, actually,
	\[\lim_{T\to \infty} \frac{1}{2T} \int_{-T}^T \log{|1 - \overline{\alpha}f(\sigma + it)|} \, dt = \int_{\mathbb{T}^\infty} \log{|1 - \overline{\alpha} f^\ast(\chi)|} \, dm_{\infty}(\chi),\]
	for every $\sigma > 0$.
\end{proof}

\begin{remark}
	In the case that $f$ is periodic, for example when $f(s) = \psi(2^{-s})$ for an inner function $\psi \in H^\infty(\mathbb{D})$, this is the classical Frostman theorem. The fact that $\phi_\alpha \circ \psi$ is inner and satisfies
	\[\lim_{\sigma \to 0^+} \int_{0}^{2\pi} \log|\phi_\alpha \circ \psi(2^{-\sigma} e^{i\theta})| \, \frac{d\theta}{2\pi} = 0\]
	is exactly saying that $\phi_\alpha \circ \psi$ has no singular inner factor, and hence is a Blaschke product.
\end{remark}

In Section~\ref{sec:prelim}, we defined in \eqref{eq:Phiw} the auxiliary function
\[\Phi_w(s) = \frac{w-\varphi(s)}{\overline{w}+\varphi(s)-1}\]
for $\varphi \in \mathscr{G}_0$ and $w\in\mathbb{C}_{1/2}$. By the mapping properties of $\varphi$, we have that $\Phi_w \in \mathscr{H}^\infty$ and $\|\Phi_w\|_{\mathscr{H}^\infty}\leq 1$. It is also clear that $\Phi_w$ is inner (for any $w \in \mathbb{C}_{1/2}$) if and only if $\mre{\varphi^\ast(\chi)}=1/2$ for almost every $\chi \in \mathbb{T}^\infty$.

Moreover, a computation reveals that the Frostman shifts of $\Phi_w$ are given by
\[\phi_\alpha \circ \Phi_w = \lambda_\alpha \Phi_{w_\alpha},\qquad\text{where}\quad  \lambda_w = -\frac{1+\alpha}{1+\overline{\alpha}}\quad\text{and} \quad w_\alpha = \frac{w+\alpha(1-\overline{w})}{1+\alpha}.\]
Note that $|\lambda_\alpha|=1$ for every $\alpha \in \mathbb{D}$. Note also that if we fix $w \in \mathbb{C}_{1/2}$, then $\alpha \mapsto w_\alpha$ is a conformal map from $\mathbb{D}$ to $\mathbb{C}_{1/2}$ which extends to a conformal map of the Riemann sphere onto itself. Hence we obtain the following from Theorem~\ref{thm:sigmainner}.

\begin{corollary} \label{cor:Phiwinner}
	Suppose that $\varphi\in\mathscr{G}_0$ is such that $\mre{\varphi^\ast(\chi)}=1/2$ for almost every $\chi \in \mathbb{T}^\infty$. Then
	\[\lim_{\sigma\to0^+} \lim_{T\to\infty} \frac{1}{2T}\int_{-T}^T \log{|\Phi_w(\sigma+it)|}\,dt=0\]
	for quasi-every $w \in \mathbb{C}_{1/2}$.
\end{corollary}

\section{A digression on the work of Jessen and Tornehave} \label{sec:JT}
Suppose that $f$ is a Dirichlet series with $\sigma_{\operatorname{u}}(f) \leq 0$. For each $\sigma > 0$, we can determine an argument $t \mapsto \arg^+ f(\sigma+it)$ which is continuous except at zeros of $f$; if $f(\sigma+it_0) = 0$ of multiplicity $j$, we define $\arg^+ f(\sigma+it)$ to have a jump of magnitude $j\pi$, corresponding to moving around the zero in a small arc to the right of it. Lagrange's mean motion problem asks if the right mean motion
\[c^+(\sigma) = \lim_{T\to\infty} \frac{\arg^+ f(\sigma + iT) - \arg^+ f(\sigma - iT) }{2T}\]
exists. This was answered in the affirmative in ground breaking work of Jessen and Tornehave \cite{JT45}, after partial results had been obtained by Bernstein, Bohl, Bohr, Hartman, Weyl,  and several others. Jessen and Tornehave solved Lagrange's problem for a wider class than that of ordinary Dirichlet series, but we limit our discussion to this setting. To simplify our discussion we will also assume that $f(+\infty) \neq 0$.

By the Cauchy--Riemann equations, there is an evident formal relationship between the Jessen function and the mean motion $c^+(\sigma)$, namely
\[c^+(\sigma) = \mathscr{J}'_f(\sigma + 0).\]
Note that the right derivative $\mathscr{J}'_f(\sigma + 0)$ exists, owing to the convexity of the Jessen function, which from our point of view is a consequence of Theorem~\ref{thm:logintdown} and Lemma~\ref{lem:Jessen}.

Similarly, recalling our assumption that $f(+\infty) \neq 0$ and applying the argument principle, the mean motion $c^+(\sigma)$ formally coincides, up to a factor of $-2\pi$, with the \emph{unweighted mean counting function}
\[\mathscr{Z}_f(\sigma) = \lim_{T \to \infty} \frac{1}{2T} \sum_{\substack{s \in f^{-1}(\{0\}) \\ \,\,|\mim{s}|<T \\ \sigma<\mre{s} < \infty}} 1,\]
that is,
\begin{equation} \label{eq:mmcounting1}
	\mathscr{Z}_f(\sigma) = - \frac{1}{2\pi} c^+(\sigma) = - \frac{1}{2\pi}\mathscr{J}'_f(\sigma + 0).
\end{equation}
Jessen and Tornehave \cite[Thm.~31]{JT45} proved that $\mathscr{Z}_f(\sigma)$ and $c^+(\sigma)$ both exist and that \eqref{eq:mmcounting1} is valid, for every $\sigma > 0$. Since the right-derivative of a convex function in an open interval must be right-continuous, we conclude the following.
\begin{lemma} \label{lem:Hfcont}
	If $\sigma_{\operatorname{u}}(f)\leq0$, then the unweighted mean counting function $\mathscr{Z}_f$ exists and is a right-continuous function of $\sigma > 0$.
\end{lemma}

In the process of solving Lagrange's mean motion problem, Jessen and Tornehave established a number of other properties of the Jessen function that we now list. None of these are strictly needed for the development of the main results of this paper, but are included for their intrinsic interest and for the fact that they are reflected in the mean counting function $\mathscr{M}_f(w, \sigma)$, see Corollary~\ref{cor:Mprop}.

We have already seen that $\mathscr{J}_f(\sigma) = \log |f(+\infty)|$ for all sufficiently large $\sigma$, in the proof of Lemma~\ref{lem:Nout} (with $a_1 = f(+\infty) \neq 0$). More generally, $\sigma \mapsto \mathscr{J}_f(\sigma)$ is linear in an interval $0 < \sigma_0 < \sigma < \sigma_1 \leq \infty$ if and only if $f$ is zero-free in the strip $\sigma_0 < \mre s < \sigma_1$. In each such linearity interval we have that $\mathscr{J}'_f(\sigma) = -\log n$, for some positive integer $n$. The number of linearity intervals in any compact subinterval of $(0, \infty)$ is always finite. For a periodic function $f(s) = \psi(2^{-s})$, it is obvious that every point $\sigma > 0$ is either a point of non-differentiability or belongs to a linearity interval of $\mathscr{J}_f$. This statement is not true for non-periodic functions, which generally exhibit complicated zero sets.

The number of points where $\mathscr{J}_f$ is not differentiable is also locally finite. By convexity, at any point $\sigma_0 > 0$ of non-differentiability, we have that
\begin{equation} \label{eq:mmcounting2}
	\lim_{\sigma \to \sigma_0^{\pm}} \mathscr{J}'_f(\sigma) = \mathscr{J}'_f(\sigma_0 \pm 0),
\end{equation}
where the right-hand side denotes the right/left derivative of the Jessen function at $\sigma_0$.  Furthermore, for any $\sigma > 0$,
\begin{equation}\label{eq:nondiff1}
	\begin{aligned} 
			\lim_{\varepsilon \to 0^+} (\mathscr{Z}_f(\sigma - \varepsilon) - \mathscr{Z}_f(\sigma + \varepsilon)) &= \frac{1}{2\pi}(\mathscr{J}'_f(\sigma + 0) - \mathscr{J}'_f(\sigma - 0)) \\
			&= \lim_{T \to \infty} \frac{1}{2T} \sum_{\substack{s \in f^{-1}(\{0\}) \\ |\mim{s}| < T \\ \mre{s} = \sigma}} 1. 
	\end{aligned}
\end{equation}
For the second equality, we refer to the proofs of Theorems~26 and 27 in \cite{JT45}. See also \cite[Corollary]{Fav08} for an explicit statement of \eqref{eq:nondiff1} in the case that $f$ is a Dirichlet polynomial. In particular, $\mathscr{J}_f$ is non-differentiable at a point $\sigma > 0$ if and only if the quantities in \eqref{eq:nondiff1} are non-zero, and there are only finitely many such points in any compact subinterval of $(0,\infty)$.

\section{The mean counting function} \label{sec:mean}
Throughout this section we assume that $f$ is a Dirichlet series with $\sigma_{\operatorname{u}}(f) \leq 0$. For $\sigma_0 > 0$, the connection between the Jessen function and our weighted mean counting function 
\begin{equation} \label{eq:Mphisigma0}
	\mathscr{M}_f(w, \sigma_0) = \lim_{T \to \infty} \frac{\pi}{T} \sum_{\substack{s \in f^{-1}(\{w\}) \\ \,\,\,\,\,|\mim{s}|<T \\ \sigma_0<\mre{s}<\infty}} \mre{s}
\end{equation}
is provided by Littlewood's lemma \cite{Littlewood24}. Littlewood's lemma is a rectangular version of Jensen's formula that is particularly well-suited to analyzing the zeros of Dirichlet series. The lemma is frequently used in analytic number theory (see \cite[Sec.~9.9]{Titchmarsh86}). Since evidently $\mathscr{M}_f(w,\sigma_0)=\mathscr{M}_{f-w}(0,\sigma_0)$, we will for simplicity formulate the next two results for $w=0$.

\begin{lemma} \label{lem:littlewoodlemma}
	Suppose that $\sigma_{\operatorname{u}}(f)\leq0$ and that $f(+\infty)\neq0$. For every $\sigma_0>0$ such that $f$ is zero-free on $\mre{s}=\sigma_0$,
	\begin{equation} \label{eq:Jensenfunc}
		\lim_{T\to\infty} \frac{\pi}{T} \sum_{\substack{s \in f^{-1}(\{0\}) \\ \,\,\,\,\,|\mim{s}|<T \\ \sigma_0<\mre{s}<\infty}} \big(\mre{s}-\sigma_0\big) = \mathscr{J}_f(\sigma_0)  - \log |f(+\infty)|.
	\end{equation}
\end{lemma}
\begin{proof}
	Fix $\sigma_0>0$ such that $f$ is zero-free on $\mre{s}=\sigma_0$. By almost periodicity there exists an increasing sequence $(T_j)_{j\geq1}$ of positive real numbers which are relatively dense in $\mathbb{R}_+$, and for which $|f(\sigma \pm iT_j)| \geq c > 0$ for some constant $c$, for all $\sigma \geq \sigma_0$. Since $f(+\infty)\neq0$, we can, by Lemma~\ref{lem:bohr} and \eqref{eq:au12}, find $\gamma>\sigma_0$ such that the Dirichlet series
	\begin{equation} \label{eq:logdifff}
		\frac{f'(s)}{f(s)} = \sum_{n=2}^\infty b_n n^{-s}
	\end{equation}
	is absolutely convergent in $\overline{\mathbb{C}_\gamma}$. In particular, $f$ does not vanish in $\overline{\mathbb{C}_\gamma}$. For $T \in (T_j)_{j\geq1}$ and $\sigma_1>\gamma$, Littlewood's lemma says that
	\begin{multline*}
		2\pi \sum_{\substack{s \in f^{-1}(\{0\}) \\ \,\,\,\,|\mim{s}| < T \\ \sigma_0 < \mre{s}<\sigma_1}} \big(\mre{s}-\sigma_0\big) = \int_{-T}^T \log|f(\sigma_0+it)|\,dt - \int_{-T}^T \log|f(\sigma_1+it)|\,dt \\
		+\int_{\sigma_0}^{\sigma_1}\arg f(\sigma+iT)\,d\sigma-\int_{\sigma_0}^{\sigma_1} \arg f(\sigma-iT)\,d\sigma.
	\end{multline*}
	Here $\arg f(s)$ refers to a continuous branch of the argument in a simply connected set containing $+\infty$.
	
	First let $\sigma_1\to\infty$. Then clearly, 
	\[\lim_{\sigma_1\to \infty}\int_{-T}^T \log|f(\sigma_1+it)|\,dt =2T\log|f(+\infty)|.\]
	For $\sigma>\gamma$, integrating term-wise, we find that
	\[|\arg f(\sigma+iT) -\arg f(\sigma-iT)| = \left|\mim \int_{-T}^T \frac{f'(\sigma + it)}{f(\sigma + it)} \, dt \right| \leq \sum_{n=2}^\infty \frac{2 |b_n|}{n^{-\sigma}\log{n}} \ll 2^{-\sigma},\]
	since \eqref{eq:logdifff} is absolutely convergent in $\overline{\mathbb{C}_\gamma}$. Thus we have that
	\[\lim_{\sigma_1 \to \infty} \left(\int_{\gamma}^{\sigma_1}\arg f(\sigma+iT)\,d\sigma-\int_{\gamma}^{\sigma_1} \arg f(\sigma-iT)\,d\sigma \right) = O(1),\]
	independently of $T$. Note that this limit certainly exists, since the left-hand side in Littlewood's lemma is constant for fixed $T$ and $\sigma_1 > \gamma$.

	Secondly, a simple argument shows that $|f'(s)| \ll 2^{-\mre s}$ for $\mre s \geq \sigma_0>0$. Thus, by our choice of $(T_j)_{j\geq1}$,
	\[\arg f(\sigma \pm iT_j) = \arg f(+\infty) - \mim \int_{\sigma}^\infty \frac{f'(\varsigma \pm iT_j)}{f(\varsigma \pm iT_j)} \, d\varsigma = O(1)\]
	for all $\sigma \geq \sigma_0$. Therefore
	\[\int_{\sigma_0}^{\gamma} \arg f(\sigma \pm iT_j)\,d\sigma = O(1),\]
	independently of $j$.

	Taking the limit $\sigma_1 \to \infty$ in Littlewood's lemma, dividing by $2T_j$, and letting $j\to\infty$, we have now demonstrated that
	\[\lim_{j \to \infty} \frac{\pi}{T_j}\sum_{\substack{s \in f^{-1}(\{0\}) \\ \,\,\,\,\,\,\,|\mim{s}| < T_j \\ \sigma_0 < \mre s<\infty}} \big(\mre{s}-\sigma_0\big) = \mathscr{J}_f(\sigma_0) - \log|f(+\infty)|.\]
	Let $d = \sup_{j\geq1} (T_{j+1} - T_j)$. To show \eqref{eq:Jensenfunc} it is now sufficient to argue that
	\begin{equation} \label{eq:Tast}
		\lim_{T^\ast \to \infty} \frac{\pi}{T^\ast}\sum_{\substack{s \in f^{-1}(\{0\}) \\ T^\ast -d <|\mim{s}| < T^\ast + d \\ \sigma_0 < \mre{s}<\infty}} \big(\mre{s}-\sigma_0\big) = 0.
	\end{equation}
	But by almost periodicity, $\{f(s + iT^\ast)\}_{T^\ast > 0}$ is a normal family in $\sigma_0 \leq \mre s \leq \gamma$ which includes no sequence convergent to zero. Thus Hurwitz's theorem shows that there is a number $N$, independent of $T^\ast$, such that $f(s + iT^\ast) = 0$ has at most $N$ zeros for $\sigma_0 \leq \mre s \leq \gamma$ and $|\mim s| < d$. Since $f$ can not have any zeros for $\mre s > \gamma$, we therefore have that
	\[\frac{\pi}{T^\ast}\sum_{\substack{s \in f^{-1}(\{0\}) \\ T^* -d <|\mim{s}| < T^* + d \\ \sigma_0 < \mre{s}<\infty}} \big(\mre{s}-\sigma_0\big) \leq \frac{2\pi N(\gamma - \sigma_0)}{T^\ast},\]
	from which \eqref{eq:Tast} follows at once.
\end{proof}

We will now bring the work of Jessen and Tornehave into play. Namely, we will use Lemma~\ref{lem:Hfcont} to extend the validity of \eqref{eq:Jensenfunc} to every $\sigma_0 > 0$, and to establish the existence of the mean counting function \eqref{eq:Mphisigma0}.

\begin{theorem} \label{thm:littlewoodlemma}
	Suppose that $\sigma_{\operatorname{u}}(f)\leq0$ and that $f(+\infty)\neq0$. For every $\sigma_0>0$ 
	\[\mathscr{M}_f(0, \sigma_0) = \lim_{T\to\infty} \frac{\pi}{T} \sum_{\substack{s \in f^{-1}(\{0\}) \\ \,\,\,\,\,|\mim{s}|<T \\ \sigma_0<\mre{s}<\infty}} \mre{s}\]
	exists and is right-continuous. Moreover,
	\begin{equation} \label{eq:countingJHrep}
		\mathscr{M}_f(0, \sigma_0) = \mathscr{J}_f(\sigma_0) - \log |f(+\infty)| + 2\pi\sigma_0\mathscr{Z}_f(\sigma_0).
	\end{equation}
\end{theorem}
\begin{proof}
	From Lemma~\ref{lem:Hfcont} we know that $\mathscr{Z}_f$ is right-continuous on $(0, \infty)$. Therefore the same is true of
	\[\mathscr{M}^+_{f}(0, \sigma) = \varlimsup_{T\to\infty} \frac{\pi}{T} \sum_{\substack{s \in f^{-1}(\{0\}) \\ \,\,|\mim{s}| < T \\ \sigma < \mre{s}<\infty}} \mre s,\]
	and
	\[\mathscr{M}^-_{f}(0, \sigma) = \varliminf_{T\to\infty} \frac{\pi}{T} \sum_{\substack{s \in f^{-1}(\{0\}) \\ \,\,|\mim{s}| < T \\ \sigma < \mre{s}<\infty}} \mre s,\]
	as can be seen from the inequalities
	\begin{align*}
		0 \leq \mathscr{M}^{\pm}_{f}(0, \sigma) - \mathscr{M}^\pm_{f}(0, \sigma + \varepsilon) &\leq \varlimsup_{T\to\infty} \frac{\pi}{T} \sum_{\substack{s \in f^{-1}(\{0\}) \\ |\mim{s}| < T \\ \,\,\,\sigma < \mre{s} \leq \sigma + \varepsilon}} \mre s \\ 
		&\leq 2\pi(\sigma + \varepsilon) (\mathscr{Z}_f(\sigma) - \mathscr{Z}_f(\sigma + \varepsilon)), \qquad \varepsilon > 0.
	\end{align*}
	Since we know from Lemma~\ref{lem:littlewoodlemma} that \eqref{eq:Jensenfunc} holds when $f$ is zero-free on $\mre s = \sigma_0$, the existence of $\mathscr{Z}_f(\sigma_0)$ implies that $\mathscr{M}^+_{f}(0, \sigma_0) = \mathscr{M}^-_{f}(0, \sigma_0)$ in this case. By the just demonstrated right continuity, this continues to hold for every $\sigma_0 > 0$, and thus
\[\mathscr{M}_f(0, \sigma) = \lim_{T\to\infty} \frac{\pi}{T} \sum_{\substack{s \in f^{-1}(\{0\}) \\ \,\,|\mim{s}| < T \\ \sigma < \mre{s}<\infty}} \mre s\]
	exists and is right-continuous for $\sigma > 0$. Approaching $\mre s = \sigma_0$ from the right on zero-free lines, and using that the Jessen function itself is continuous, we conclude that \eqref{eq:Jensenfunc} holds for every $\sigma_0 > 0$, which is \eqref{eq:countingJHrep}.
\end{proof}

By \eqref{eq:countingJHrep}, \eqref{eq:mmcounting2}, and the continuity of the Jessen function, we see that if $\sigma_{\operatorname{u}}(f)\leq0$, $f(+\infty)\neq0$ and $\sigma>0$, then
\[\lim_{\varepsilon \to 0^+} \big(\mathscr{M}_f(0, \sigma- \varepsilon) - \mathscr{M}_f(0, \sigma+\varepsilon)\big) = \sigma\big(\mathscr{J}'_{f}(\sigma + 0) - \mathscr{J}'_{f}(\sigma - 0)\big).\]
It is also self-evident that $\mathscr{M}_f(0, \sigma)$ is a non-increasing function of $\sigma$. By the differentiability properties of the Jessen function outlined in Section~\ref{sec:JT}, we therefore obtain the following result, after recalling that $\mathscr{M}_f(w,\sigma_0)=\mathscr{M}_{f-w}(0,\sigma)$.

\begin{corollary} \label{cor:Mprop}
	Suppose that $\sigma_{\operatorname{u}}(f)\leq 0$ and that $w \neq f(+\infty)$. Then $\mathscr{M}_f(w, \sigma)$ is non-increasing and right-continuous in $\sigma > 0$. The number of discontinuities is finite in any compact subset of $(0,\infty)$, and the points of discontinuity are precisely the points of non-differentiability of the Jessen function $\mathscr{J}_{f-w}.$
\end{corollary}

We now prove the existence of the mean counting function for $f \in \mathscr{N}_{\operatorname{u}}$.
\begin{theorem} \label{thm:Mexistgeneral}
	Let $f \in \mathscr{N}_{\operatorname{u}}$ and $w \neq f(+\infty)$. Then the limits
	\[\mathscr{M}_f(w) = \lim_{\sigma_0 \to 0^+} \mathscr{M}_{f}(w, \sigma_0) = \lim_{\sigma_0 \to 0^+}  \mathscr{J}_{f- w}(\sigma_0) - \log |f(+\infty) - w|\]
	exist and are equal.
\end{theorem}

\begin{proof}
	By the inequality $\log^+|f-w| \leq |w| + \log^+|f|$, we have that $f - w \in \mathscr{N}_{\operatorname{u}}$. We may therefore assume without loss of generality that $w=0$ and that $f(+\infty)\neq0$. Since $\mathscr{J}_f(\sigma_0)$ is non-increasing in $\sigma_0$, we conclude that $\lim_{\sigma_0\to 0^+} \mathscr{J}_f(\sigma_0)$ exists, by Lemma~\ref{lem:Nu}. By Theorem~\ref{thm:littlewoodlemma}, we therefore have that
	\[0 \leq \mathscr{M}_{f}(0, \sigma_0) - 2\pi\sigma_0 \mathscr{Z}_f(\sigma_0) \leq \lim_{\sigma_0 \to 0^+} \mathscr{J}_f(\sigma_0) -  \log |f(+\infty)| < \infty,\]
	for all $\sigma_0>0$. Since $\mre{s}< 2 (\mre{s}-\sigma_0/2)$ for $\mre{s}>\sigma_0$, we find that
	\begin{align*}
		\mathscr{M}_{f}(0, \sigma_0) &< 2 \big(\mathscr{M}_{f}(0, \sigma_0/2) - 2\pi\sigma_0 \mathscr{Z}_f(\sigma_0/2)\big) \\ &\leq 2\left(\lim_{\sigma_0 \to 0^+} \mathscr{J}_f(\sigma_0) -  \log |f(+\infty)|\right) < \infty.
	\end{align*}
	Therefore
	\[\mathscr{M}_f(0) = \lim_{\sigma_0 \to 0^+} \mathscr{M}_f(0,\sigma_0)\]
	exists, since $\mathscr{M}_f(0,\sigma_0)$ is obviously non-increasing in $\sigma_0$.
	
	By Theorem~\ref{thm:littlewoodlemma} it remains to prove that
	\[\lim_{\sigma_0 \to 0^+} \sigma_0 \mathscr{Z}_f(\sigma_0)=0.\]
	For $\varepsilon>\sigma_0 >0$ such that $f$ is zero-free on $\mre{s}=\varepsilon$, we write
	\[\sigma_0 \mathscr{Z}_f(\sigma_0) = \sigma_0 \mathscr{Z}_f(\varepsilon) + \lim_{T \to \infty} \frac{1}{2T} \sum_{\substack{s \in f^{-1}(\{0\}) \\ \,\,\,\,\,\,\,\,|\mim{s}| < T \\ \sigma_0 < \mre{s} < \varepsilon}} \sigma_0.\]
	Clearly,
	\[\lim_{T \to \infty} \frac{1}{2T} \sum_{\substack{s \in f^{-1}(\{0\}) \\ \,\,\,\,\,\,\,\,|\mim{s}| < T \\ \sigma_0 < \mre{s} < \varepsilon}} \sigma_0\leq \frac{2}{\pi}(\mathscr{M}_f(0, \sigma_0) - \mathscr{M}_f(0, \varepsilon)).\]
	We have already seen that the right-hand side converges to zero as $\varepsilon \to 0^+$, uniformly in $0 < \sigma_0 < \varepsilon$. Hence, by first choosing $\varepsilon$ small and then letting $\sigma_0 \to 0^+$, we see that $\lim_{\sigma_0 \to 0^+} \sigma_0 \mathscr{Z}_f(\sigma_0) = 0.$
\end{proof}

The following result is needed in the proof of Theorem~\ref{thm:compact}. 

\begin{lemma} \label{lem:submean}
	Suppose that $f\in\mathscr{N}_{\operatorname{u}}$. The mean counting function $\mathscr{M}_f$ satisfies the submean value property,
	\begin{equation} \label{eq:submean}
		\mathscr{M}_f(w) \leq \frac{1}{\pi r^2} \int_{\mathbb{D}(w,r)} \mathscr{M}_f(s)\,ds,
	\end{equation}
	for every open disc $\mathbb{D}(w,r)$ which does not contain $f(+\infty)$.
\end{lemma}

\begin{proof}
	By Lemma~\ref{lem:Jessen}, we have for $\sigma>0$ that
	\[\mathscr{J}_{f-w}(\sigma) = \int_{\mathbb{T}^\infty} \log|f_\chi(\sigma)-w|\,dm_\infty(\chi).\]
	Since $w \mapsto \log|f_\chi(\sigma)-w|$ is subharmonic for fixed $\chi\in\mathbb{T}^\infty$, it follows from Fubini's theorem that
	\begin{equation} \label{eq:Jsubmean}
		\mathscr{J}_{f-w}(\sigma) \leq \frac{1}{\pi r^2} \int_{\mathbb{D}(w,r)} \mathscr{J}_{f-s}(\sigma)\,ds
	\end{equation}
	for every disc $\mathbb{D}(w,r)$. By Theorem~\ref{thm:Mexistgeneral}, we have that
	\[\lim_{\sigma\to0^+} \mathscr{J}_{f-s}(\sigma) = \mathscr{M}_f(s) + \log|f(+\infty)-s|\]
	for every $s \neq f(+\infty)$. Furthermore, 
\[\mathscr{J}_{f-s}(\sigma) \geq  \log|f(+\infty)-s|\]
for every $\sigma > 0$. Since $\sigma \mapsto \mathscr{J}_{f-s}(\sigma)$ is non-increasing in $\sigma$, we may therefore apply the monotone convergence theorem to \eqref{eq:Jsubmean}. If $\mathbb{D}(w,r)$ does not contain $f(+\infty)$ this gives us \eqref{eq:submean}, by the harmonicity of $s \mapsto \log|f(+\infty)-s|$ in this disc. 
\end{proof}

We are now in a position to prove the first of our main results.

\begin{proof}[Proof of Theorem~\ref{thm:mean}]
	Theorem~\ref{thm:Mexistgeneral} directly gives the existence of the mean counting function for symbols $\varphi \in \mathscr{G}_0$ in the Gordon--Hedenmalm class, since Lemma~\ref{lem:map0p1} implies that $\mathscr{G}_0 \subset \mathscr{N}_{\operatorname{u}}$. We give an alternative argument here, which also yields the point-wise estimate \eqref{eq:littlewoodineqM}. For a point $w \in \mathbb{C}_{1/2}\setminus\{\varphi(+\infty)\}$, consider as in \eqref{eq:Phiw} the Dirichlet series
	\[\Phi_w(s) = \frac{w-\varphi(s)}{\overline{w}+\varphi(s)-1},\]
	and observe that
	\[\mathscr{M}_\varphi(w)  = \mathscr{M}_{\Phi_w}(0) = \lim_{\sigma \to 0^+}  \mathscr{M}_{\Phi_w}(0, \sigma).\]
	Here we have noted that $\Phi_w \in \mathscr{H}^\infty \subset \mathscr{N}_{\operatorname{u}}$ and applied Theorem~\ref{thm:Mexistgeneral}. Because $\Phi_w$ maps $\mathbb{C}_0$ to $\mathbb{D}$, we see furthermore that
	\[\mathscr{M}_\varphi(w) = \lim_{\sigma \to 0^+} \mathscr{J}_{\Phi_w}(\sigma) - \log|\Phi_w(+\infty)| \leq -\log|\Phi_w(+\infty)|.\]
	Since
	\[-\log|\Phi_w(+\infty)| = \log\left|\frac{\overline{w}+\varphi(+\infty)-1}{w-\varphi(+\infty)}\right|\]
	this finishes the proof.
\end{proof}

Suppose that $\varphi \in \mathscr{G}_0$ with $\varphi(+\infty) = \nu \in \mathbb{C}_{1/2}$, and let $\varphi_\nu$ be the extremal symbol of \eqref{eq:varphinu}. Then $\|\mathscr{C}_\varphi f\|_{\mathscr{H}^2} \leq \|\mathscr{C}_{\varphi_\nu} f\|_{\mathscr{H}^2}$ for every $f\in\mathscr{H}^2$ (see \cite[Sec.~2]{Brevig17}). In \cite[Thm.~21]{BP20}, it was proven that the following are equivalent:
\begin{enumerate}
	\item[(a)] $\mre \varphi^\ast(\chi) = 1/2$ for almost every $\chi \in \mathbb{T}^\infty$.
	\item[(b)] $\|\mathscr{C}_\varphi f\|_{\mathscr{H}^2} = \|\mathscr{C}_{\varphi_\nu} f\|_{\mathscr{H}^2}$ for every $f \in \mathscr{H}^2$.
	\item[(c)] $\|\mathscr{C}_\varphi\|_{\mathscr{H}^2} = \|\mathscr{C}_{\varphi_\nu}\|_{\mathscr{H}^2}$.
\end{enumerate}
Condition (a) is perhaps less tangible than it may seem. Indeed, in \cite[Thm.~1]{SS09} a function $g \in \mathscr{H}^\infty$ is constructed for which $\|g\|_{\mathscr{H}^\infty} = 1$ and $|g(it)| = 1$ for almost every $t \in \mathbb{R}$, but such that $g$ is not inner, in the sense that $|g^\ast(\chi)| < 1$ on a set of positive measure in $\mathbb{T}^\infty$. Here $g(it)$ refers to the usual boundary values of $g$ in $\mathbb{C}_0$,
\[g(it) = \lim_{\sigma \to 0^+} g(\sigma + it).\]
This is easily translated to an example of a function $\varphi \in \mathscr{G}_0$ such that $\mre \varphi(it) = 1/2$ for almost every $t \in \mathbb{R}$, but for which $\mre \varphi^\ast(\chi) > 1/2$ on a set of positive measure.

Our second main result provides an alternative to condition (a) in terms of the mean counting function.

\begin{proof}[Proof of Theorem~\ref{thm:sharpiro}]
	We begin by using Lemma~\ref{lem:Jessen}, Theorem~\ref{thm:Mexistgeneral}, and the method of proof of Theorem~\ref{thm:mean}, to establish that
	\begin{equation} \label{eq:sharpiropf}
		\mathscr{M}_\varphi(w) =  \log\left|\frac{\overline{w}+\nu-1}{w-\nu}\right|  + \lim_{\sigma \to 0^+} \int_{\mathbb{T}^\infty} \log |(\Phi_w)_\chi(\sigma)| \, dm_\infty(\chi)
	\end{equation}
	for every $\varphi \in \mathscr{G}_0$ with $\varphi(+\infty)=\nu$ and every $w \in \mathbb{C}_{1/2}\setminus\{\nu\}$. 
	
	Suppose that $\mre{\varphi^\ast(\chi)}=1/2$ for almost every $\chi \in \mathbb{T}^\infty$. Then Corollary~\ref{cor:Phiwinner} and \eqref{eq:sharpiropf} show that
\[\mathscr{M}_\varphi(w) = \log\left|\frac{\overline{w}+\nu-1}{w-\nu}\right|\]
for quasi-every $w \in \mathbb{C}_{1/2}$. 

	Conversely, suppose that $\mathscr{M}_\varphi(w) = \log\big|\frac{\overline{w}+\nu-1}{w-\nu}\big|$ for some $w \in \mathbb{C}_{1/2}$. Since $|(\Phi_w)_\chi (\sigma)| \leq 1$ for every $\chi\in\mathbb{T}^\infty$ and all $\sigma>0$, Fatou's lemma and \eqref{eq:sharpiropf} hence imply that
	\[\int_{\mathbb{T}^\infty} \log |\Phi_w^\ast(\chi)| \, dm_\infty(\chi) \geq \lim_{\sigma \to 0^+} \int_{\mathbb{T}^\infty} \log |(\Phi_w)_\chi(\sigma)| \, dm_\infty(\chi) = 0.\]
	Therefore $|\Phi_w^\ast(\chi)| = 1$ for almost every $\chi$, that is, $\Phi_w$ is inner. This is equivalent to the fact that $\mre \varphi^\ast(\chi) = 1/2$ for almost every $\chi \in \mathbb{T}^\infty$.
\end{proof}

We close this section by noting that simple modifications of the proofs of Theorems~\ref{thm:mean} and \ref{thm:sharpiro} also yield the following analogue of the classical Littlewood inequality.
\begin{theorem} \label{thm:innertfae}
	Suppose that $f \in \mathscr{H}^\infty$ with $\|f\|_{\mathscr{H}^\infty} \leq 1$ and that $\varphi(+\infty) = \nu$ for some $\nu \in \mathbb{D}$. Then the mean counting function $\mathscr{M}_\varphi$ exists for every $\xi \in \mathbb{D}\setminus\{\nu\}$ and enjoys the point-wise estimate
	\[\mathscr{M}_f(\xi) \leq \log\left|\frac{1-\overline{\xi}\nu}{\xi-\nu}\right|. \]
	Furthermore, the following are equivalent:
	\begin{enumerate}
		\item[(i)] $f$ is inner.
		\item[(ii)] $\mathscr{M}_f(\xi) = \log\big|\frac{1-\overline{\xi}\nu}{\xi-\nu}\big|$ for quasi-every $\xi \in \mathbb{D}$.
		\item[(iii)] $\mathscr{M}_f(\xi) = \log\big|\frac{1-\overline{\xi}\nu}{\xi-\nu}\big|$ for one $\xi \in \mathbb{D}$.
	\end{enumerate}
\end{theorem}

\begin{remark}
	In the special case when $f$ is periodic, that is, when $f(s)=\psi(2^{-s})$ for some $\psi \in H^\infty(\mathbb{D})$ with $\psi(0)=\nu$, Theorem~\ref{thm:innertfae} is precisely the Littlewood inequality \eqref{eq:littlewoodineq} and the corresponding result for the Nevanlinna counting function \eqref{eq:nevanlinnadisc} \cite[Sec.~4.2]{Shapiro87}.
\end{remark}

\section{Compact composition operators} \label{sec:compact}
The first step toward the proof of Theorem~\ref{thm:compact} is to deduce the change of variable formula of Theorem~\ref{thm:cov}. Recall that if $\varphi\in\mathscr{G}_0$, then 
\[\mathscr{M}_\varphi(w,\sigma)=\lim_{T\to\infty}\frac{\pi}{T}\sum_{\substack{s \in \varphi^{-1}(\{w\}) \\ \,\,|\mim{s}|<T \\ \sigma<\mre{s}<\infty}}\mre s \]
exists for $w\neq\varphi(+\infty)$ and $\sigma>0$. Moreover, the mean counting function $\mathscr{M}_\varphi(w)$ is obtained by letting $\sigma\to0^+$,
\[\mathscr{M}_\varphi(w) = \lim_{\sigma \to 0^+} \mathscr{M}_\varphi(w,\sigma).\]
Note also that $\mathscr{M}_\varphi(w,\sigma)\to0$ as $\sigma\to\infty$, since $\varphi(s)=w$ lacks solutions for $s$ with sufficiently large real part, seeing as $w\neq\varphi(+\infty)$. To justify taking the limit $T\to\infty$ in the change of variables formula, we will use the uniform estimates of Lemma~\ref{lem:ThetaLsub}.

\begin{proof}[Proof of Theorem~\ref{thm:cov}]
	Fix $\varphi \in \mathscr{G}_0$. Suppose first that $f$ is a Dirichlet polynomial and that $0<\sigma_0<\sigma_1<\infty$. A non-injective change of variables, as in Section~\ref{sec:prelim}, then yields that
	\[\frac{2}{T} \int_{-T}^T \int_{\sigma_0}^{\sigma_1} |(f\circ\varphi)'(s)|^2 \,dt \, \sigma d\sigma = \frac{2}{\pi}\int_{\mathbb{C}_{1/2}} |f'(w)|^2 \,\frac{\pi}{T} \sum_{\substack{s \in \varphi^{-1}(\{w\}) \\ \,\,\,\,|\mim{s}|<T \\ \sigma_0<\mre{s} \leq \sigma_1}}\mre s\,dw.\]
	The dominated convergence theorem together with the estimate \eqref{eq:ThetaLsub2} of Lemma~\ref{lem:ThetaLsub} shows that
	\begin{multline*}
		\lim_{T\to\infty} \frac{2}{\pi}\int_{\mathbb{C}_{1/2}\setminus \mathbb{D}(\varphi(+\infty),\delta)} |f'(w)|^2 \,\frac{\pi}{T} \sum_{\substack{s \in \varphi^{-1}(\{w\}) \\ \,\,\,\,|\mim{s}|<T \\ \sigma_0<\mre{s} \leq \sigma_1}}\mre s\,dw \\
		= \frac{2}{\pi}\int_{\mathbb{C}_{1/2}\setminus \mathbb{D}(\varphi(+\infty),\delta)} |f'(w)|^2 \, \big(\mathscr{M}_\varphi(w,\sigma_0)-\mathscr{M}_\varphi(w,\sigma_1)\big)\,dw
	\end{multline*}
	for every $\delta>0$. To handle the points close to $\varphi(+\infty)$, let $0<\delta< (1/2+ \mre \varphi(+\infty))/2$, so that $\mathbb{D}(\varphi(+\infty),\delta) \subset \mathbb{C}_{1/2}$. For $T > \sigma_1$, we use the estimate \eqref{eq:ThetaLsub1} of Lemma~\ref{lem:ThetaLsub} and a point-wise estimate for $|f'(w)|^2$ to find that
	\begin{multline*}
		\frac{2}{\pi}\int_{\mathbb{D}(\varphi(+\infty),\delta)} |f'(w)|^2 \,\frac{\pi}{T} \sum_{\substack{s \in \varphi^{-1}(\{w\}) \\ \,\,\,\,|\mim{s}|<T \\ \sigma_0<\mre{s}\leq\sigma_1}}\mre s\,dw  \\ \ll \|f\|_{\mathscr{H}^2}^2 \int_{\mathbb{D}(\varphi(+\infty),\delta)} \log \left|\frac{\overline{w}+ \varphi(2T)-1}{w - \varphi(2T)} \right|\,dw,
	\end{multline*}
	where the implied constant only depends on $\mre{\varphi(+\infty)}$. For $T>\sigma_1>0$, it is clear that $\varphi(2T)$ is contained in a bounded subset of $\mathbb{C}_{1/2}$. Hence the integral on the right hand side is bounded by $\delta^2 \log{\delta^{-1}}$ as $\delta \to 0^+$, and we conclude that
	\begin{multline*}
		\lim_{T\to\infty}\frac{2}{T} \int_{-T}^T \int_{\sigma_0}^{\sigma_1} |(f\circ\varphi)'(s)|^2 \,dt \sigma d\sigma \\
		= \frac{2}{\pi}\int_{\mathbb{C}_{1/2}} |f'(w)|^2 \, \big(\mathscr{M}_\varphi(w,\sigma_0)-\mathscr{M}_\varphi(w,\sigma_1)\big)\,dw
	\end{multline*}
We now let $\sigma_1\to \infty$ and $\sigma_0\to 0^+$ in this equation. Since $\sigma_{\operatorname{u}}(f\circ \varphi)\leq0$ we may as in Lemma~\ref{lem:curlyLP} work out the limit of the left hand side by calculating with coefficients. On the right hand side we simply apply the monotone convergence theorem.  The result, as in \eqref{eq:curlyLP}, is that
	\begin{equation*}
		\|\mathscr{C}_\varphi f\|_{\mathscr{H}^2}^2 - |f(\varphi(+\infty))|^2 = \frac{2}{\pi}\int_{\mathbb{C}_{1/2}} |f'(w)|^2 \mathscr{M}_\varphi(w)\,dw.
	\end{equation*}
To extend this formula to general $f \in \mathscr{H}^2$ we approximate with polynomials, helped by the fact that $\mathscr{C}_\varphi \colon \mathscr{H}^2 \to \mathscr{H}^2$ is bounded.
\end{proof}

In order to deduce Theorem~\ref{thm:compact} from Theorem~\ref{thm:cov}, several estimates are required. We begin with the following basic result, which demonstrates that it is only the behavior of the mean counting function $\mathscr{M}_\varphi(w)$ near $\mre{w}=1/2$ which is relevant for the compactness of $\mathscr{C}_\varphi$.

\begin{lemma} \label{lem:weakuniform}
	Suppose that $(f_j)_{j\geq1}$ is a sequence in $\mathscr{H}^2$ such that $\|f_j\|_{\mathscr{H}^2}=1$ for every $j\geq1$ and such that $f_j$ converges weakly to $0$. For every $\theta>1/2$ and $\varepsilon>0$ there is some $J=J(\theta,\varepsilon)$ such that
	\[|f_j(\varphi(+\infty))|^2+\frac{2}{\pi} \int_{\mre{w}\geq \theta} |f_j'(w)|^2 \,\mathscr{M}_\varphi(w)\,dw \leq 2\varepsilon^2, \qquad j \geq J.\]
\end{lemma}

\begin{proof}
	Since $f_j$ converges weakly to $0$, there is $J=J(\varepsilon)$ such that $|f_j(\varphi(+\infty))|\leq \varepsilon$ for $j\geq J$. Set $e_2(s) = 2^{-s}$. If we can prove that there is some $J=J(\theta,\varepsilon)$ such that
	\begin{equation} \label{eq:e2lol}
		|f_j'(s)| \leq \varepsilon |e_2'(s)|
	\end{equation}
	for $\mre{s}\geq \theta$ and $j\geq J$, then also
	\[\frac{2}{\pi} \int_{\mre{w}\geq \theta} |f_j'(w)|^2 \,\mathscr{M}_\varphi(w)\,dw \leq \frac{2}{\pi} \int_{\mathbb{C}_{1/2}} \varepsilon^2 |e_2'(w)|^2 \mathscr{M}_\varphi(w)\,dw \leq \varepsilon^2 \|\mathscr{C}_\varphi e_2\|_{\mathscr{H}^2}^2 \leq \varepsilon^2.\]
	 Let us therefore prove \eqref{eq:e2lol}. Write $f_j(s) = \sum_{n\geq1} a_n(j) n^{-s}$. By the Cauchy--Schwarz inequality and the fact that $\|f_j\|_{\mathscr{H}^2}=1$ we get for every $N\geq1$ the estimate
	\begin{equation} \label{eq:tailest}
		\left|\sum_{n=N}^\infty a_n(j) (\log{n})n^{-s}\right| \leq \left(\sum_{n=N}^\infty \frac{(\log{n})^2}{n^{2\sigma}}\right)^{\frac{1}{2}}.
	\end{equation}
	Since $|e_2'(s)|= (\log{2}) 2^{-\sigma}$ and since $\theta>1/2$, we see that there is some $N=N(\theta,\varepsilon)$ such that the left hand side of \eqref{eq:tailest} is bounded by $(\varepsilon/2) |e_2'(s)|$ for $\mre{s}\geq \theta$. Since $(f_j)_{j\geq1}$ converges weakly to $0$, there is $J=J(N,\varepsilon)$ such that $|a_n(j)| \leq \varepsilon/(2N)$ holds for $n=1,\ldots,N$ whenever $j\geq J$. We then find that
	\begin{equation} \label{eq:faceest}
		\left|\sum_{n=1}^N a_n(j) (\log{n}) n^{-s} \right| \leq \frac{\varepsilon}{2N} \sum_{n=1}^N \frac{\log{n}}{n^{\sigma}} \leq \frac{\varepsilon}{2} |e_2'(s)|
	\end{equation}
	for $\mre{s} > 1/2$. We obtain the desired estimate \eqref{eq:e2lol} from \eqref{eq:tailest}, \eqref{eq:faceest} and the triangle inequality.
\end{proof}

Next we describe a class of Carleson measures for $\mathscr{H}^2$.

\begin{lemma} \label{lem:carleson}
	Suppose that $\omega \colon [0,\infty) \to [0,\infty)$ is decreasing and in $L^1 \cap L^\infty$. There is an absolute constant $C$ such that the embedding inequality
	\begin{equation} \label{eq:carleson}
		\int_{-\infty}^\infty |f(1/2+it)|^2 \,\omega(|t|)\,dt \leq C\left(\|\omega\|_{L^\infty}+\|\omega\|_{L^1}\right) \|f\|_{\mathscr{H}^2}^2
	\end{equation}
	for every $f \in \mathscr{H}^2$.
\end{lemma}

\begin{proof}
	The local embedding \cite[Thm.~4.11]{HLS97} states that there is a constant $C_1$, independent of $\tau\in\mathbb{R}$, such that
	\begin{equation} \label{eq:localemb}
		\int_{\tau}^{\tau+1} |f(1/2+it)|^2 \,dt \leq C_1 \|f\|_{\mathscr{H}^2}^2.
	\end{equation}
	 Define $\omega_k(t) = \omega(k) \mathbf{1}_{[k,k+1)}(t)$ for $k=0,1,2,\ldots$. Since $\omega$ is decreasing, we find that
	\[\int_0^\infty |f(1/2+it)|^2 \,\omega(t)\,dt \leq \sum_{k=0}^\infty \omega(k) \int_{k}^{k+1} |f(1/2+it)|^2\, dt \leq C_1 \|f\|_{\mathscr{H}^2}^2 \sum_{k=0}^\infty \omega(k)\]
	where $C_1$ is as in \eqref{eq:localemb}. Using again that $\omega$ is decreasing, we find that
	\[\sum_{k=0}^\infty \omega(k) \leq \omega(0) + \int_0^\infty \omega(t)\,dt = \|\omega\|_{L^\infty}+\|\omega\|_{L^1}.\]
	By symmetry we obtain \eqref{eq:carleson} with $C=2C_1$.
\end{proof}

\begin{remark}
	An interesting open problem is to describe the non-negative weights $\omega$ for which there is a finite constant $C(\omega)$ such that
	\[\int_{-\infty}^\infty |f(1/2+it)|^2 \omega(t)\,dt \leq C(\omega) \|f\|_{\mathscr{H}^2}^2.\]
	It is necessary that $\omega \in L^1(\mathbb{R})\cap L^\infty(\mathbb{R})$, as can be seen from for example \cite[Sec.~2]{OS12}.
\end{remark}

We will use the following consequence of Lemma~\ref{lem:carleson} in the proof of Theorem~\ref{thm:compact}. To appreciate its relevance, recall from the Littlewood--type inequality of Theorem~\ref{thm:mean} and the upper bound in Lemma~\ref{lem:logest} (with $\nu=\varphi(+\infty)$) that for every $\varphi \in \mathscr{G}_0$, the counting function $\mathscr{M}_\varphi$ satisfies the estimate
\begin{equation} \label{eq:Llogest}
	\mathscr{M}_\varphi(w) \leq \log\left|\frac{\overline{w}+\varphi(+\infty)-1}{w-\varphi(+\infty)}\right| \leq 2\frac{(\mre{w}-1/2)(\mre{\varphi(+\infty)}-1/2)}{|w-\varphi(+\infty)|^2}
\end{equation}
for $w \in \mathbb{C}_{1/2}\setminus\{\varphi(+\infty)\}$.
\begin{lemma} \label{lem:carleson2}
	Let $\nu \in \mathbb{C}_{1/2}$ and $\delta>0$. Then there is a constant $C=C(\delta)$ such that 
	\[\int_{1/2<\mre{w}<\theta} |f'(w)|^2 \frac{\mre{w}-1/2}{|w-\nu|^{1+\delta}}\,dw \leq \frac{C}{(\mre{\nu}-\theta)^{1+\delta}} \|f\|_{\mathscr{H}^2}^2\]
	for every $f\in\mathscr{H}^2$ and $1/2<\theta<\mre{\nu}$.
\end{lemma}

\begin{proof}
	Since both the $\mathscr{H}^2$-norm and the strip $1/2<\mre{w}<\theta$ are invariant under vertical translations, we may without loss of generality assume that $\nu$ is real. On the vertical line $1/2<\mre{w}=\sigma<\theta$, we apply Lemma~\ref{lem:carleson} with
	\[\omega(t) = \frac{1}{|\sigma+it-\nu|^{1+\delta}} \leq \frac{1}{|\theta+it-\nu|^{1+\delta}}\]
	to the Dirichlet series $g(s)=f'(\sigma-1/2+s)$, yielding that
\[\int_{-\infty}^\infty \frac{|f'(\sigma + it)|^2}{|\sigma+it-\nu|^{1+\delta}} \,dt \ll  \frac{1}{(\mre{\nu}-\theta)^{1+\delta}} \sum_{n=1}^\infty |a_n|^2 \frac{(\log n)^2}{n^{2(\sigma-1/2)}}\]
The proof is completed upon multiplying by $(\sigma - 1/2)$ and integrating in $\sigma$, since
	\[\int_{1/2}^\theta \frac{(\log{n})^2}{n^{2(\sigma-1/2)}}\,(\sigma-1/2)\,d\sigma \leq \int_{1/2}^\infty \frac{(\log{n})^2}{n^{2(\sigma-1/2)}}\,(\sigma-1/2)\,d\sigma = \frac{1}{4}. \qedhere\]
\end{proof}

We are finally ready to proceed with the proof of our last main result. We will rely on Theorem~\ref{thm:cov}, Lemma~\ref{lem:submean}, Lemma~\ref{lem:weakuniform}, and Lemma~\ref{lem:carleson2}.

\begin{proof}[Proof of Theorem~\ref{thm:compact}]
	Let $\varphi \in \mathscr{G}_0$ and suppose that 
		\begin{equation} \label{eq:compact}
	\lim_{\mre w \to \frac{1}{2}^+} \frac{\mathscr{M}_{\varphi}(w)}{\mre w - 1/2} = 0.
	\end{equation}
	We need to show that $\mathscr{C}_\varphi$ is compact. Fix some $0<\delta<1$. We begin with following claim. For every $\varepsilon>0$ there is some $1/2<\theta \leq(1/2+\mre{\varphi(+\infty)})/2$ such that if $1/2 <\mre{w}<\theta$, then
	\begin{equation} \label{eq:claim}
		\mathscr{M}_\varphi(w) \leq \varepsilon^2 \frac{(\mre{w}-1/2)}{|w-\varphi(+\infty)|^{1+\delta}}.
	\end{equation}
	Suppose that the claim does not hold. Then we can find a constant $c>0$ and a sequence $(w_j)_{j\geq1}$ in $\mathbb{C}_{1/2}$ such that $\mre{w_j}\to 1/2^+$ and
	\begin{equation} \label{eq:notclaim}
		\mathscr{M}_\varphi(w_j) \geq c \frac{(\mre{w_j}-1/2)}{|w_j-\varphi(+\infty)|^{1+\delta}}.
	\end{equation}
	If $(\mim{w_j})_{j\geq1}$ is bounded, this is in contradiction with \eqref{eq:compact}. On the other hand, if $(\mim{w_j})_{j\geq1}$ is unbounded, then \eqref{eq:notclaim} contradicts the Littlewood--type inequality \eqref{eq:Llogest}, since $0<\delta<1$. Consequently, \eqref{eq:claim} holds.
	
	We will now apply the claim. The implied constants below depend only on $\mre{\varphi(+\infty)}$ and $\delta$. For every $f \in \mathscr{H}^2$, we get from \eqref{eq:claim} that
	\[\frac{2}{\pi} \int_{1/2 < \mre{w}< \theta} |f'(w)|^2 \,\mathscr{M}_\varphi(w)\,dw \ll \varepsilon^2 \int_{1/2<\mre{w}<\theta} |f'(w)|^2\, \frac{(\mre{w}-1/2)}{|w-\varphi(+\infty)|^{1+\delta}}\,dw.\]
	Using Lemma~\ref{lem:carleson2} with $\nu=\varphi(+\infty)$ and $1/2<\theta \leq(1/2+\mre{\varphi(+\infty)})/2<\mre{\nu}$, we conclude that
	\begin{equation} \label{eq:Cfinal}
		\frac{2}{\pi} \int_{1/2 < \mre{w}< \theta} |f'(w)|^2 \,\mathscr{M}_\varphi(w)\,dw \ll \varepsilon^2 \|f\|_{\mathscr{H}^2}^2.
	\end{equation}
	Suppose now that $(f_j)_{j\geq1}$ is a sequence in $\mathscr{H}^2$ such that $\|f_j\|_{\mathscr{H}^2}=1$ for every $j\geq1$ and such that $f_j$ converges weakly to $0$. To see that $\mathscr{C}_\varphi$ is compact, we need to show that $\|\mathscr{C}_\varphi f_j\|_{\mathscr{H}^2}\to 0$ as $j \to \infty$. Fix $\varepsilon>0$ and let $\theta=\theta(\varepsilon)>1/2$ be such that \eqref{eq:Cfinal} holds. By Lemma~\ref{lem:weakuniform} we can find $J=J(\theta,\varepsilon)$ such that
	\[|f_j(\varphi(+\infty))|^2 + \frac{2}{\pi} \int_{\mre{w}\geq \theta} |f_j'(w)|^2 \,\mathscr{M}_\varphi(w)\,dw \leq 2\varepsilon^2\]
	whenever $j\geq J$. Combining this with \eqref{eq:Cfinal} we find that 
	\[|f_j(\varphi(+\infty))|^2 +\frac{2}{\pi} \int_{\mathbb{C}_{1/2}} |f_j'(w)|^2 \,\mathscr{M}_\varphi(w)\,dw \ll \varepsilon^2\]
	for $j\geq J$. Hence, by Theorem~\ref{thm:cov}, we conclude that $\|\mathscr{C}_\varphi f_j\|_{\mathscr{H}^2} \ll \varepsilon$ when $j\geq J$ and consequently $\mathscr{C}_\varphi$ is compact.
	
	Suppose now that $\mathscr{C}_\varphi$ is compact. We need to show that \eqref{eq:compact} holds. Let $(w_j)_{j\geq1}$ be any sequence in $\mathbb{C}_{1/2}$ such that $\mre w_j \to 1/2^+$ as $j \to \infty$. We may without loss of generality assume that 
	\begin{equation} \label{eq:recond}
		\mre{w_j} \leq \frac{1/2+\mre{\varphi(+\infty)}}{2}.
	\end{equation}
	Set $r_j = (\mre{w_j}-1/2)/2$. The condition \eqref{eq:recond} ensures that $\varphi(+\infty)$ is uniformly bounded away from the discs $\mathbb{D}(w_j,r_j)\subset \mathbb{C}_{1/2}$. By Lemma~\ref{lem:submean}, we therefore have
	\begin{equation} \label{eq:meanrj}
		\frac{\mathscr{M}_\varphi(w_j)}{\mre{w_j}-1/2} \leq \frac{1}{2\pi r_j^3} \int_{\mathbb{D}(w_j,r_j)} \mathscr{M}_\varphi(s)\,ds.
	\end{equation}
	The normalized reproducing kernel of $\mathscr{H}^2$ at $w \in \mathbb{C}_{1/2}$ is 
	\[K_w(s) = \frac{\zeta(s+\overline{w})}{\sqrt{\zeta(2\mre{w})}}.\]
	The sequence $(K_{w_j})_{j\geq1}$ converges weakly to $0$ in $\mathscr{H}^2$. Since $\zeta(s)=(s-1)^{-1}+E(s)$ for an entire function $E$, there is a constant $C>0$ such that
	\[|K_{w_j}'(s)|^2 \geq \frac{C}{r_j^3}\] 
	for every $s \in \mathbb{D}(w_j,r_j)$. Inserting this estimate into \eqref{eq:meanrj} and extending the integral to $\mathbb{C}_{1/2}$, we find by Theorem~\ref{thm:cov} that
	\[\frac{\mathscr{M}_\varphi(w_j)}{\mre{w_j}-1/2} \leq \frac{1}{2\pi C} \int_{\mathbb{D}(w_j,r_j)} |K_{w_j}'(s)|^2 \mathscr{M}_\varphi(s)\,ds \leq \frac{1}{4C} \|\mathscr{C}_\varphi K_{w_j} \|_{\mathscr{H}^2}^2.\]
	Furthermore, by the assumption of compactness, we have that $\mathscr{C}_\varphi K_{w_j} \to 0$ in $\mathscr{H}^2$ as $j \to \infty$. It follows that 
	\[\lim_{j\to\infty} \frac{\mathscr{M}_\varphi(w_j)}{\mre{w_j}-1/2} = 0,\]
	and, since $(w_j)_{j \geq 1}$ was arbitrary, that \eqref{eq:compact} holds.
\end{proof}

\section{Concluding remarks and further work} \label{sec:further}

\subsection{Interchanging the limits} \label{subsec:limswap}
The question left unaddressed in this paper is whether it is possible to interchange the limits $T \to \infty$ and $\sigma \to 0^+$ in the definition \eqref{eq:meancounting} of the mean counting function.
\begin{problem} 
	Let $f \in \mathscr{N}_{\operatorname{u}}$. Is it true that
	\[\mathscr{M}_f(w) = \lim_{T\to\infty} \frac{\pi}{T} \sum_{\substack{s \in f^{-1}(\{w\}) \\ \,\,|\mim{s}|<T \\ 0<\mre{s}<\infty}} \mre{s}\]
	for every $w \neq f(+\infty)$?
\end{problem}
In the same way that we proved \eqref{eq:limswap}, it is easy to see that the answer is affirmative if $f$ is of the form $f(s) = \psi(2^{-s})$ for some $\psi$ belonging to the usual Nevanlinna class $N$ of $\mathbb{D}$. By the work in Section~\ref{sec:mean}, the answer to the problem is also yes if $\sigma_{\operatorname{u}}(f) < 0$. We also note that inequality \eqref{eq:ThetaLsub1} of Lemma~\ref{lem:ThetaLsub} shows that
\[ \varlimsup_{T\to\infty} \frac{\pi}{T} \sum_{\substack{s \in \varphi^{-1}(\{w\}) \\ \,\,|\mim{s}|<T \\ 0<\mre{s}<\infty}} \mre{s} \leq C \log \left|\frac{\overline{w}+ \varphi(+\infty)-1}{w - \varphi(+\infty)} \right|\]
for every $\varphi \in \mathscr{G}_0$ and $w \neq \varphi(+\infty)$, where $C$ is an absolute constant.

\subsection{Composition operators in Schatten classes} \label{subsec:schatten} 
Let $T$ be a bounded linear operator on a Hilbert space $H$. Define the $n$th approximation number $a_n(T)$ as the distance in the operator norm from $T$ to the operators of rank $<n$. Then $T$ is compact if and only if $a_n(T)\to 0$ as $n\to\infty$. The Schatten class $S_p$ consists of the operators $T\colon H\to H$ such that 
\[\|T\|_{S_p}^p = \sum_{n=1}^\infty (a_n(T))^p<\infty.\]

The approximation numbers of composition operators $\mathscr{C}_\varphi \colon \mathscr{H}^2\to \mathscr{H}^2$ have been investigated in the series of papers \cite{Bayart03,BB17,FQV04,QS15} and the survey \cite{Queffelec15}. The relevant results of \cite{Bayart03} and \cite{FQV04} mainly pertain to the membership of $\mathscr{C}_\varphi$ to the Hilbert--Schmidt class $S_2$ for certain polynomial symbols $\varphi$. The paper \cite{QS15} contains general lower bounds on the decay of $a_n(\mathscr{C}_\varphi)$, incorporating a counting function that is not of mean type. This leads to fairly precise estimates for the approximation numbers of composition operators generated by affine symbols
\[\varphi(s) = c_1 + \sum_{j=1}^d c_j p_j^{-s}.\]
In \cite[Sec.~8]{BB17} more general polynomial symbols are also considered within the framework developed in \cite{QS15}. We mention only \cite[Cor.~4]{BB17} which states that for every pair $0<p<q<\infty$, there is a polynomial symbol $\varphi$ such that $\mathscr{C}_\varphi\in S_q\setminus S_p$.
 
Recall that if $(x_n)_{n\geq1}$ is an orthonormal basis for the Hilbert space $H$ and $T\colon H\to H$ is a bounded linear operator, then
\begin{equation} \label{eq:hsnorm}
	\|T\|_{S_2}^2 = \sum_{n=1}^\infty \|T x_n\|_H^2.
\end{equation}
We have the following corollary of Theorem~\ref{thm:cov}, which describes the Hilbert--Schmidt norm of a composition operator on $\mathscr{H}^2$ in terms of $\mathscr{M}_\varphi$.
\begin{corollary} 
	Suppose that $\varphi\in\mathscr{G}_0$. Then
	\begin{equation} \label{eq:hilbertschmidt}
		\|\mathscr{C}_\varphi\|_{S_2}^2 = \zeta(2\mre \varphi(+\infty))+\frac{2}{\pi}\int_{\mathbb{C}_{1/2}}\zeta''(2\mre{w})\,\mathscr{M}_\varphi(w)\,dw,
	\end{equation}
	where $\zeta''(s) = \sum_{n\geq1} (\log{n})^2 n^{-s}$.
\end{corollary}

\begin{proof}
	We want to use \eqref{eq:hsnorm} with $H=\mathscr{H}^2$ and $x_n(s) = n^{-s}$. By Theorem~\ref{thm:cov},
\begin{align*}	
\|\mathscr{C}_\varphi\|_{S_2}^2 &= \sum_{n=1}^\infty \|\mathscr{C}_\varphi x_n\|_{\mathscr{H}^2}^2 \\
&= \zeta(2\mre \varphi(+\infty)) +\sum_{n=2}^\infty \,\frac{2}{\pi}\int_{\mathbb{C}_{1/2}} \,\left|(-\log{n}) n^{-w}\right|^2 \,\mathscr{M}_\varphi(w)\,dw,
\end{align*}
	which yields \eqref{eq:hilbertschmidt} after changing the order of summation and integration.
\end{proof}

It would be interesting to know if the techniques of the present paper can be applied to the following.

\begin{problem} 
	Describe the symbols $\varphi\in\mathscr{G}_0$ such that $\mathscr{C}_\varphi\in S_p$.
\end{problem}

For composition operators on $H^2(\mathbb{D})$, membership in Schatten classes has been classified in terms of the Nevanlinna counting function by Luecking and Zhu \cite{LZ92}. Their result states that if $\phi \colon \mathbb{D}\to\mathbb{D}$ is an analytic function, then the composition operator generated by $\phi$ on $H^2(\mathbb{D})$ is in $S_p$ if and only if 
\begin{equation} \label{eq:LZ}
	\int_{\mathbb{D}} \frac{(N_\phi(\xi))^{p/2}}{(1-|\xi|^2)^{p/2+2}}\,d\xi < \infty,
\end{equation}
where $N_\phi$ denotes the Nevanlinna counting function \eqref{eq:nevanlinnadisc}. Since
\[\zeta''(2\mre{w}) = \frac{1}{4(\mre{w}-1/2)^3}+O(1)\]
for, say, $1/2<\mre{w}<3/2$, we see that the integrability condition on $N_\phi$ in \eqref{eq:LZ} with $p=2$ is analogous to the (local) integrability condition on $\mathscr{M}_\varphi$ in \eqref{eq:hilbertschmidt}.

\subsection{Compactness in case (b) of $\mathscr{G}$} \label{subsec:b}
In the present paper, we have been exclusively focused on symbols associated with case (a) of the Gordon--Hedenmalm class $\mathscr{G}$. For the purpose of comparison with Theorem~\ref{thm:compact}, we close the paper by compiling some known results for case (b). We recall that the symbols in this case are of the form
\[\varphi(s)= c_0 s + \varphi_0(s),\]
where $c_0$ is a positive integer and either $\varphi_0 \equiv 0$ or $\varphi_0$ is a Dirichlet series mapping $\mathbb{C}_0$ to $\mathbb{C}_0$ with $\sigma_{\operatorname{u}}(\varphi_0)\leq 0$. If $\varphi_0\equiv0$, then $\mathscr{C}_\varphi$ is an isometry, and so this possibility will not be considered further.

A key difference between the cases (a) and (b), is that the symbol $\varphi$ is not almost periodic when $c_0\geq1$. In analogy with the Nevanlinna counting function \eqref{eq:nevanlinnadisc}, Bayart \cite{Bayart03} introduced the counting function
\begin{equation} \label{eq:nevanlinnac0}
	N_\varphi(w) = \sum_{s \in \varphi^{-1}(\{w\})} \mre{s}.
\end{equation}
for $\varphi\in\mathscr{G}$ with $c_0\geq1$. By use of conformal mappings from $\mathbb{C}_0$ to $\mathbb{D}$, it was established in \cite[Prop.~3]{Bayart03} that the counting function \eqref{eq:nevanlinnac0} satisfies the estimate
\begin{equation} \label{eq:bayartlittlewood}
	N_\varphi(w) \leq \frac{\mre{w}}{c_0}
\end{equation}
for every $w \in \mathbb{C}_0$. The estimate \eqref{eq:bayartlittlewood} is best compared with the case $\phi(0)=0$ in the Littlewood inequality \eqref{eq:littlewoodineq}, since $\varphi(+\infty)=+\infty$ for every $\varphi$ in $\mathscr{G}$ with $c_0\geq1$. Note that $\mre{w}$ is the distance from $w$ to the imaginary axis.

In analogy with \eqref{eq:compactdisc} and Theorem~\ref{thm:compact}, it seems plausible that the compactness of $\mathscr{C}_\varphi$ on $\mathscr{H}^2$ could be related to the condition that $N_\varphi(w) = o(\mre{w})$ as $\mre{w}\to 0^+$. In the following result, the first statement is from \cite[Thm.~2]{Bayart03} and the second statement is from the proof of \cite[Thm.~6]{Bailleul15}, which in turn relies on \cite[Thm.~3]{Bayart03}.

We say that a Dirichlet series $f(s) = \sum_{n\geq1} a_n n^{-s}$ is finitely generated if the set $\{n\,:\,a_n\neq0\}$ is multiplicatively generated by a finite set of prime numbers.

\begin{theorem}
	Suppose that $\varphi\in\mathscr{G}$ with $c_0\geq1$.
	\begin{enumerate}
		\item[(i)] Suppose that $\mim \varphi_0$ is bounded. If $N_\varphi(w) = o(\mre{w})$ as $\mre{w}\to 0^+$, then $\mathscr{C}_\varphi$ is compact on $\mathscr{H}^2$.
		\item[(ii)] Suppose that $\varphi$ is finitely valent and that $\varphi_0$ is finitely generated. If $\mathscr{C}_\varphi$ is compact on $\mathscr{H}^2$, then $N_\varphi(w) = o(\mre{w})$ as $\mre{w}\to 0^+$.
	\end{enumerate}
\end{theorem}
The condition that $\mim\varphi_0$ is bounded seems to be primarily technical. Unlike the situation in the proof of Theorem~\ref{thm:compact}, the upper bound of the Littlewood--type inequality \eqref{eq:bayartlittlewood} does not provide additional decay of the counting function $N_\varphi(w)$ as $|\mim{w}|\to \infty$. The requirement that $\varphi_0$ be finitely generated is related to the fact that the reproducing kernels of $\mathscr{H}^2$ do not generally converge in $\mathbb{C}_0$.

In view of Theorem~\ref{thm:compact}, to completely resolve the compactness problem discussed in \cite[Prob.~4]{Hedenmalm04} and \cite[Prob.~3.3]{SS16}, it remains to solve the following.

\begin{problem}
	Classify the symbols $\varphi\in\mathscr{G}$ which generate compact composition operators on $\mathscr{H}^2$ in the case that $c_0\geq1$.
\end{problem}

\bibliographystyle{amsplain} 
\bibliography{compact} 
\end{document}